\documentclass[10pt]{article}
\usepackage[utf8]{inputenc}

\usepackage{amsmath}
\usepackage{amssymb}
\usepackage{amsthm}
\usepackage{mathrsfs}

\usepackage{graphicx}
\graphicspath{ {/desktop/graphs/} }

\usepackage{version}
\usepackage[margin=3cm]{geometry}
\usepackage{color}
\usepackage{mathtools}
\usepackage{tikz}
\usetikzlibrary{arrows}
\usetikzlibrary{trees}
\usepackage{float}

\numberwithin{equation}{section}

\newtheorem{claim*}{Claim}[section]

\newtheorem{definition}{Definition}[section]
\newtheorem{lemma}{Lemma}[section]
\newtheorem{theorem}{Theorem}[section]
\newtheorem{corollary}{Corollary}[section]

\newtheorem{remark}{Remark}
\newcommand\numberthis{\addtocounter{equation}{1}\tag{\theequation}}

\DeclareMathOperator{\N}{\mathbb{N}}
\DeclareMathOperator{\E}{\mathbb{E}}

\DeclareMathOperator{\Q}{\mathcal{Q}}
\DeclareMathOperator{\Pb}{\mathbb{Pb}}

\DeclareMathOperator{\VAR}{Var}

\DeclareMathOperator{\PAR}{Par}

\DeclareMathOperator{\BM}{\mathbb{B}{\rm m}}

\newcommand{\vb}{\vspace{3mm}}

\usepackage[T1]{fontenc}

\DeclarePairedDelimiter{\ceil}{\lceil}{\rceil}

\title{Transient analysis of one-sided L\'evy-driven queues}
\author{
N. Starreveld, R. Bekker, and M. Mandjes
}
\date{\today}

\begin{document}

\maketitle

\begin{abstract}
\noindent
In this paper we analyze the transient behavior of the workload process in a L\'evy input queue. We are interested in the value of the workload process at a random epoch; this epoch is distributed as the sum of independent exponential random variables. We consider both cases of spectrally one-sided L\'evy input processes, for which we succeed in deriving explicit results. As an application we approximate the mean and the Laplace transform of the workload process after a deterministic time.
\newline \mbox{} \newline

\vb

\noindent
{\bf Affiliations:} N. Starreveld is with Korteweg-de Vries Institute for Mathematics, Science Park 904, 1098 {\sc xh} Amsterdam, 
University of Amsterdam, the Netherlands. {\it Email}: {\tt N.J.Starreveld@uva.nl}. 

\noindent
R. Bekker is with Department of Mathematics,
VU University Amsterdam,
De Boelelaan 1081a, 1081 {\sc hv} Amsterdam, The Netherlands.
{\it Email}: {\tt r.bekker@vu.nl}.

\noindent
M. Mandjes is with Korteweg-de Vries Institute for Mathematics,
University of Amsterdam, Science Park 904, 1098 {\sc xh} Amsterdam, the Netherlands. He is also affiliated with
{\sc Eurandom}, Eindhoven University of Technology, Eindhoven, the Netherlands, and CWI, Amsterdam the Netherlands. 
{\it Email}: {\tt m.r.h.mandjes@uva.nl}.

\noindent
The research of N.\ Starreveld and M.\ Mandjes is partly funded by the NWO Gravitation project NETWORKS, grant number 024.002.003.

\end{abstract}

\section{Introduction}
This paper studies the transient workload in a queue fed by a L\'evy input process 
$X=\{X_t\}_{t\geq0}$; here the workload process, in the sequel  
denoted by $\{\Q_t\}_{t\geq0}$, is defined as the reflection of $X$ at zero. 
This workload process can be constructed from the input process $X$ as the (unique) 
solution of the so-called Skorokhod problem, \cite{Man, Skor1, Skor2}. It turns out that the process
$\Q$ follows from $X$ through
\[\Q_t = X_t + \max\{\Q_0,\mathcal{L}_t\},\]
where
\[\mathcal{L}_t:=\sup_{0\leq s\leq t}-X_s = - \inf_{0\leq s\leq t}X_s.\]
The process $\{\mathcal{L}_t\}_{t\geq0}$ is often referred to as \textit{local time (at zero)} or \textit{regulator process} \cite{Kyp}. 

As mentioned above, we are interested in the transient behavior of the {workload process}. In queueing theory transient analysis is a classical topic that is treated in various standard textbooks; see e.g.\ \cite{Asm, Cohen, Pra}. Typically, transient analysis is important in situations where the time horizon considered is relatively short, so that it cannot be ensured that the system is `close to stationarity'. 
In addition, transient results are useful in cases that the net-input process changes over time; it for instance facilitates the analysis of systems with time-varying demand as well as the assessment of the impact of specific workload control mechanisms. In general, transient analysis allows us to assess the impact of the initial state ${\mathcal Q}_0$. 

The main contribution of this paper is the generalization of the existing results on the transient behavior of the workload process $\{\Q_t\}_{t\geq0}$. We consider $n$ exponentially distributed random variables $T_1,\ldots,T_n$ with parameters $q_1,\ldots,q_n$ and we analyze the joint behavior of the vector
\[\left(\Q_{T_1},\Q_{T_1+T_2},\ldots,\Q_{T_1+\ldots+T_n}\right),\] with a specific focus on $\Q_{T_1+\ldots+T_n}$. 
It is noted that this also directly yields $\Q_T$ when $T$ follows a {\it Coxian} distribution; see Section~\ref{sec:concl} for some additional background on this claim. This observation is particularly useful owing to the fact that any distribution on the positive half line can be approximated arbitrarily closely by 
 a sequence of Coxian distributions (in the sense of convergence; see e.g. \cite[Section III.4]{Asm}).
For the case of a spectrally positive input process, the results are given in terms of the Laplace-Stieltjes transform (LST), whereas we find an expression for the associated density for the spectrally negative case.
%We study the Laplace-Stieltjes transform (LST) instead of the process itself, since more explicit results can be derived. Using Laplace inversion techniques \cite{Aba}, information about the process can then be inferred from the LST. 

Apart from the general results obtained, a second contribution lies in the reasoning behind our proofs. More specifically, our proofs reveal that the above formulas obey an elegant and simple tree structure. The transient workload behavior consists of $2^n$ terms that can be recursively evaluated. We prove our results by induction; given that we know the expression for the quantity under consideration at $n-1$ exponential epochs, we derive the expression at $n$ exponential epochs. In this induction step, from $n-1$ to $n$ that is, it can be seen how each term produces two offsprings, thus giving insight into the underlying structure. The idea behind the proofs yields a \textit{mechanism} to address questions related to transient analysis at random epochs, which may help in obtaining a deeper understanding of the behavior of the underlying continuous-time queueing system.

Transient analysis of queueing systems started with the analysis of the waiting times in the M/M/1 queue \cite{Pra}. In \cite{Ben, Tak} the authors analyze the Laplace-Stieltjes transform of the waiting time process in the M/G/1 queue. The argument used there is also applied in \cite{Man}, so as  to derive Theorem \ref{Spectrally Positive One} below for the case of a compound Poisson input process. The transient analysis of L\'evy driven queues is of a much more recent date; see e.g.\ \cite{Asm, Man, Kel} for results on
the workload process in a L\'evy-driven queue at an exponential epoch (which are briefly summarized in Section~\ref{sec:pre}). As a direct application the authors of \cite{Box} study clearing models,  where special attention is paid to clearings at exponential epochs (relying on  results on  the workload at an exponential epoch in an M/G/1 setting).

Concerning the structure of the paper, in Section 2 we present our notation, as well as the preliminaries that are needed in order to prove our results.  In Section 3 we present the main results of the paper, which are Thms.~\ref{MainTheorem} and \ref{spectrally negative general case density}. We support the final results with intuitive arguments based on a tree structure; the proofs can also be interpreted along those lines.
In Section 4 we present results obtained in numerical experiments. Section 5 contains the proofs of Thms.~\ref{MainTheorem} and \ref{spectrally negative general case density}. In all sections, the spectrally positive and spectrally negative cases are treated separately. Finally, Section~\ref{sec:concl} contains conclusions and a brief discussion.

\section{Model, Notation, and Preliminaries} \label{sec:pre}

In this section we present the workload at an exponential epoch for queues with spectrally positive (Section~\ref{subs:pre-pos}) and spectrally negative (Section~\ref{subs:pre-neg}) L\'evy input processes. These results are heavily relied upon throughout the paper, and in addition serve as a benchmark. In passing, we also introduce our notation. 

\subsection{Spectrally positive L\'evy processes} \label{subs:pre-pos}

As mentioned, the building block of this paper is a L\'evy process $X=\{X_t\}_{t\geq0}$. In case $X$ is a spectrally positive process, henceforth denoted by $X\in\mathscr{S}_+$, the \textit{Laplace exponent} $\phi(\alpha):=\log\E e^{-\alpha X_1}$ is well defined for all $\alpha \geq0$. By applying H\"{o}lder's inequality we get that $\phi(\cdot)$ is convex on $[0,+\infty)$ with slope $\phi'(0)=-\E X_1$ at the origin. In general, the inverse function $\psi(\cdot)$ is not well defined and we work with the right inverse
\[\psi(q):=\sup\{\alpha\geq0: \phi(\alpha)=q\}.\]
For the case the drift of our driving process $X$ is negative we observe that $\psi'(0)=-\E X_1>0$ and thus $\phi(\cdot)$ is increasing on $[0,+\infty)$. In this case the inverse function $\psi(\cdot)$ is well defined. 

Our interest is in the transient behavior of the \textit{workload process} $\{\Q_t\}_{t\geq0}$. We consider an exponentially distributed random variable $T$ with parameter $q$ (sampled independently from the L\'evy input process) and focus on the transform $\E_x e^{-\alpha \Q_T}$, where $\alpha\geq0$ and $x$ denotes the initial workload. In this case the transform $\E_x e^{-\alpha \Q_T}$ is explicitly known, and is given in the following theorem \cite{Man, Kel, Tak}.

\begin{theorem}\label{Spectrally Positive One}
Let $X\in\mathscr{S}_+$ and let $T$ be exponentially distributed with parameter $q$, independently of $X$. For $\alpha\geq0$, $x\geq0$,
\[\E_x e^{-\alpha\Q_T}=\int_0^\infty q e^{-q t}\E_x e^{-\alpha\Q_t}{\rm d}t=\frac{q}{q-\phi(\alpha)}\left(e^{-\alpha x}-\frac{\alpha}{\psi(q)}e^{-\psi(q)x}\right).\]
\end{theorem}

Using Laplace inversion techniques \cite{Aba}, information about the process can then be inferred from the LST as it uniquely determines the distribution of $\Q_t$, for each $t$ and any initial workload $x$.

\subsection{Spectrally negative L\'evy processes} \label{subs:pre-neg}

For a spectrally negative  L\'evy process $X$, henceforth denoted by $X\in\mathscr{S}_-$, we define the \textit{cumulant} $\Phi(\beta):=\log\E e^{\beta X_1}$. This function is well-defined and finite for all $\beta\geq0$, exactly because there are no positive jumps. We observe that $\Phi(\cdot)$ has slope $\Phi'(0)=\E X_1$ at the origin, thus $\Phi(\beta)$ in general is no bijection on $[0,+\infty)$. We define the right inverse through
\[\Psi(q):=\sup\{\beta\geq0:\Phi(\beta)=q\}.\]
%Define $\beta_0:=\Psi(0)$, which plays an important role when analyzing queues with spectrally negative input.

When working with spectrally negative L\'evy processes, the so-called $q$-scale functions, $W^{(q)}(\cdot)$ and $Z^{(q)}(\cdot)$ play a crucial role, particularly when studying the fluctuation properties of the reflected process \cite{Man, Pal}. For $q\geq0$, let $W^{(q)}(x)$, for $x\geq0$, be a strictly increasing and continuous function whose Laplace transform satisfies
\begin{equation}
\label{W function}
\int_0^\infty e^{-\beta x}W^{(q)}(x){\rm d}x = \frac{1}{\Phi(\beta)-q}, \qquad \beta>\Psi(q);
\end{equation}
we let $W^{(q)}(x)$ equal 0 for $x<0$. From \cite[Th.\ 8.1.(i)]{Kyp} it follows that such a function exists. Having defined the function $W^{(q)}(\cdot)$, we define the function $Z^{(q)}(\cdot)$ as
\begin{equation}
\label{Z scale function}
Z^{(q)}(x) := 1+q\int_0^xW^{(q)}(y){\rm d}y.
\end{equation}

We immediately see the importance of the $q$-scale function in the density of the \textit{workload process} at an exponential epoch, given in the following theorem \cite[Section 4.2]{Man}, which is originally due to Pistorius \cite{Mar}.
\begin{theorem}
\label{spectrally negative one exponential time}
Let $X\in \mathscr{S}_-$ and let $T$ be exponentially distributed with parameter $q$, independently of $X$. For $\alpha\geq0$, $x\geq0$ and $\beta>0$,
\[\Pb_x(\Q_T\in{\rm d}y) = \left(e^{-\Psi(q)y}\Psi(q)Z^{(q)}(x)-qW^{(q)}(x-y)\right){\rm d}y,\]
and
\[\int_0^\infty e^{-\beta x}\E_x e^{-\alpha\Q_T}{\rm d}x =\frac{1}{\beta}\left(\frac{\Psi(q)}{\Psi(q)+\alpha}+\frac{q}{\Phi(\beta)-q}\frac{\Psi(q)-\beta}{\Psi(q)+\alpha}\frac{\alpha}{\alpha+\beta}\right).\]
\end{theorem}
The result on the LST of $\Q_T$ in the above theorem follows for $\beta>\Psi(q)$ by a direct computation from the density of $\Q_T$; by a standard analytic continuation argument the resulting expression then holds for any $\beta>0.$

\section{Main Results}

In this section we present our main results, viz.\ Thms.~\ref{MainTheorem} and \ref{spectrally negative general case density}. In both subsections we first derive the workload behavior at two exponential epochs as this clearly demonstrates how the various terms appear. Then we elaborate on the mechanism for obtaining the workload at $n$ exponential epochs, yielding an intuitively appealing tree structure. The proofs can be interpreted along those lines, and are given in full detail in Section~\ref{sec:proofs}.  

\subsection{Spectrally positive case}\label{section spectrally positive}

Suppose we have a spectrally positive L\'evy process $X$. We want to describe the behavior of the workload process $\{\Q_t\}_{t\geq0}$ at consecutive exponential epochs. We do this by considering exponentially distributed random variables $T_1,\ldots,T_n$ with distinct parameters $q_1,\ldots,q_n$ and calculate, for $\alpha_i\geq0$ and some initial workload $x\geq0$, the joint Laplace transform given by
\begin{equation}
\label{Joint Transform}
\E_x e^{-\alpha_1\Q_{T_1}-\alpha_2\Q_{T_1+T_2}+\ldots+\alpha_n\Q_{T_1+\ldots+T_n}}.
\end{equation}
 
It is instructive to first illustrate how to derive an expression for the joint transform  at two exponential epochs, i.e for $\E_x e^{-\alpha_1\Q_{T_1}-\alpha_2\Q_{T_1+T_2}}$.  From Thm.\ \ref{Spectrally Positive One} we have an expression for the transform $\E_x e^{-\alpha \Q_T}$. Consider now two exponentially distributed random variables with parameters $q_1, q_2$. Then, conditioning on $\Q_{T_1}$ in combination with applying Thm.\  \ref{Spectrally Positive One} twice, yields
\begin{align*}
\E_x e^{-\alpha_1\Q_{T_1}-\alpha_2\Q_{T_1+T_2}}=&\int_0^\infty e^{-\alpha_1y}\E_y e^{-\alpha_2\Q_{T_2}}\Pb_x(\Q_{T_1}\in{\rm d}y)\\
=&\int_0^\infty e^{-\alpha_1y}\left(\frac{q_2}{q_2-\phi(\alpha_2)}\left(e^{-\alpha_2 y}-\frac{\alpha_2}{\psi(q_2)}e^{-\psi(q_2)y}\right)\right)\Pb_x(\Q_{T_1}\in{\rm d}y)\\
=&\hspace{1mm}\frac{q_2}{q_2-\phi(\alpha_2)}\left(\E_x e^{-(\alpha_1+\alpha_2)\Q_{T_1}}-\frac{\alpha_2}{\psi(q_2)}\E_x e^{-(\alpha_1+\psi(q_2))\Q_{T_1}}\right)\\
=&\hspace{1mm}\frac{q_2}{q_2-\phi(\alpha_2)}\bigg(\frac{q_1}{q_1-\phi(\alpha_1+\alpha_2)}\left(e^{-(\alpha_1+\alpha_2)x}-\frac{\alpha_1+\alpha_2}{\psi(q_1)}e^{-\psi
(q_1)x}\right)\\
&-\frac{\alpha_2}{\psi(q_2)}\frac{q_1}{q_1-\phi(\alpha_1+\psi(q_2))}\left(e^{-(\alpha_1+\psi(q_2))x}-\frac{\alpha_1+\psi(q_2)}{\psi(q_1)}e^{-\psi(q_1)x}\right)\bigg). \numberthis \label{case2}
\end{align*}

We see that by conditioning on the value of the workload at the first exponential epoch we can derive the transform at two exponential epochs. The above reasoning rests on the property that the process $\{\Q_t\}_{t\geq0}$ is a Markov process. 

Some special attention is needed for the case $\alpha_1=0$ and $q_1=q_2$, i.e., when $T$ has an Erlang-2 distribution. From the last term in (\ref{case2}) we see that an additional limiting argument is required. A~straightforward application of `l'H\^{o}pital' then yields the expression for $\E_x e^{-\alpha \Q_T}$ as in \cite[Section 4.1]{Man}.

\bigskip
The main idea for the case of $n$ exponentially distributed random variables $T_i$ is very similar: condition on the workload at the first exponential epoch, thus obtaining
\begin{equation}
\label{relation between n and n-1}
\E_xe^{-\alpha_1\Q_{T_1}-\alpha_2\Q_{T_1+T_2}-...-\alpha_n\Q_{T_1+...+T_n}}=\int_0^\infty e^{-\alpha_1 y}\E_ye^{-\alpha_2\Q_{T_2}-...-\alpha_n\Q_{T_2+...+T_n}}\Pb_x(\Q_{T_1}\in{\rm d}y).
\end{equation}
Eqn.~(\ref{relation between n and n-1}) is used in combination with Thm.\ \ref{Spectrally Positive One} to determine the transform at $n$ exponential epochs given the joint transform at $n-1$ epochs. 
At this point, it is useful to understand how the coefficients in the exponential terms of the transform appear; this is illustrated in Fig.~\ref{tree diagram 1} below. Specifically, due to the integration in (\ref{relation between n and n-1}) and Thm.\ \ref{Spectrally Positive One}, it follows that each term produces two new terms when an exponential epoch is added (i.e., when moving from $n-1$ to $n$ exponential epochs), such that the transform at $n$ exponential epochs consists of $2^{n}$ exponential terms. We observe that in the expression for $n$ random variables the first term is, for every $n$, $\exp[{-(\alpha_1+\ldots+\alpha_n)x}]$ (multiplied by some coefficient). The exponents $\exp[{-(\alpha_1+\ldots+\alpha_{l-1}+\psi(q_l))x}]$, where $l=1,\ldots,n$, produce one exponential term of higher order $\exp[{-(\alpha_1+\ldots+\alpha_{l}+\psi(q_{l+1}))x}]$ as well as one term corresponding to  $\exp[{-\psi(q_1)x}]$; it is seen that the latter terms always appear at the `even positions'. 
This mechanism is depicted in the tree diagram in Fig.~\ref{tree diagram 1}, where row $n$ shows the $2^n$ factors when we have $n$ exponentially distributed random variables $T_1,...,T_n$. For ease we only write the exponent at every node, hence the node $\alpha_1+\psi(q_2)$ represents the term corresponding to $\exp[{-(\alpha_1+\psi(q_2))x}]$ (multiplied by some coefficient). 
In every row, the factors are counted from the left.

% Set the overall layout of the tree
\tikzstyle{level 1}=[level distance=1.5cm, sibling distance=8cm]
\tikzstyle{level 2}=[level distance=1.5cm, sibling distance=4cm]
\tikzstyle{level 3}=[level distance=1.5cm, sibling distance=2cm]
\tikzstyle{level 4}=[level distance=1.5cm, sibling distance=2cm]

% Define styles for bags and leafs
\tikzstyle{bag} = [text width=5.5em, text centered]
\tikzstyle{end} = [circle, minimum width=7pt,fill, inner sep=0pt]

\begin{figure}[H]
\begin{tikzpicture}[grow=down, sloped]
\node[bag] {}
    child {
        node[bag] {$\alpha_1$}
            child {
                node[bag] {$\alpha_1+\alpha_2$}
                     child{
                        node[bag] {$\alpha_1+\alpha_2+\alpha_3$}
                           {}
                        edge from parent
                node[above] {$$}
                node[below]  {$$}
                     }
                     child{
                       node[bag] {$\psi(q_1)$}
                     edge from parent
                node[above] {$$}
                node[below]  {$$}
                     }
                edge from parent
                node[above] {$$}
                node[below]  {$$}
            }
            child {
                node[bag] {$\psi(q_1)$}
                    child{
                        node[bag] {$\alpha_1+\psi(q_2)$}
                        edge from parent
                        node[above] {$$}
                        node[below]  {$$}
                     }
                     child{
                       node[bag] {$\psi(q_1)$}
                     edge from parent
                node[above] {$$}
                node[below]  {$$}
                     }
                edge from parent
                node[above] {$$}
                node[below]  {$$}
            }
            edge from parent
            node[above] {}
            node[below]  {$$}
    }
    child {
        node[bag] {$\psi(q_1)$}
            child {
                node[bag] {$\alpha_1+\psi{(q_2)}$}
                     child{
                        node[bag] {$\alpha_1+\alpha_2+\psi(q_3)$}
                          {}
                        edge from parent
                        node[above] {$$}
                        node[below]  {$$}
                     }
                     child{
                       node[bag] {$\psi(q_1)$}
                     edge from parent
                node[above] {$$}
                node[below]  {$$}
                     }
                edge from parent
                node[above] {$$}
                node[below]  {$$}
            }
            child {
                node[bag] {$\psi(q_1)$}
                  child{
                        node[bag] {$\alpha_1+\psi(q_2)$}
                        edge from parent
                node[above] {$$}
                node[below]  {$$}
                     }
                  child{
                       node[bag] {$\psi(q_1)$}
                     edge from parent
                node[above] {$$}
                node[below]  {$$}
                     }
                edge from parent
                node[above] {$$}
                node[below]  {$$}
            }
            edge from parent
            node[above] {}
            node[below]  {$$}
    };

\end{tikzpicture}
\caption{The exponents in (\ref{Joint Transform}) at every step}
\label{tree diagram 1}
\end{figure}
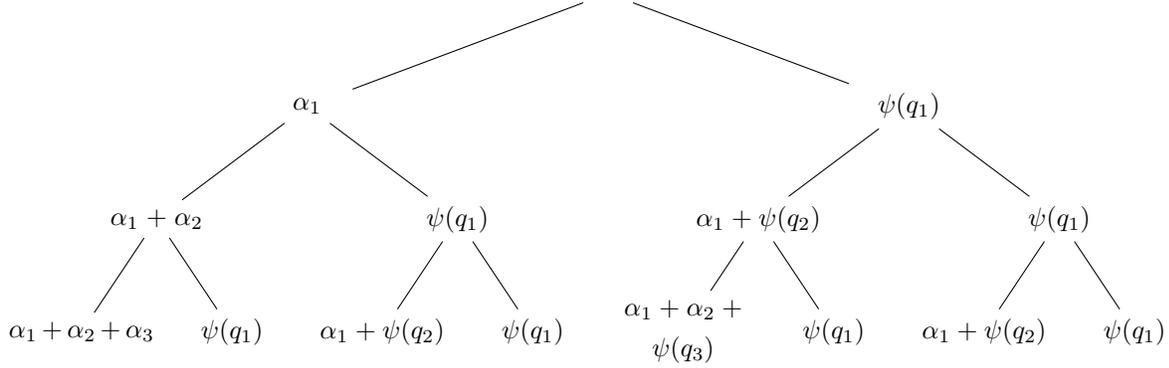

We observe that the entire tree consists of subtrees starting from a node $\psi(q_1)$ (apart from the first element of every row). Suppose we have the element $\exp[{-(\alpha_1+\ldots+\alpha_{l-1}+\psi(q_l))x}]$ in the $n$-th row. This originates from a subtree generated by an initial node $\psi(q_1)$ that is $l-1$ rows higher in the tree. This follows from the fact that if we start from the node $\psi(q_1)$ we have to move $l-1$ times down and left in order to reach the node $\alpha_1+\ldots+\alpha_{l-1}+\psi(q_l)$.    
So the node $\alpha_1+\ldots+\alpha_{l-1}+\psi(q_l)$ in the $n$-th row belongs to a subtree spanned from the node $\psi(q_1)$ in the $(n-l+1)$-th row. For the ordering of terms, we assume that this initial node is at position $2j$ for some $j=1,2,\ldots,2^{n-l}$; we recall here that the nodes $\psi(q_1)$ are located at the even positions of each row. Since the node is at position $2j$ there are $2j-1$ nodes in front of it. At every step downwards in the tree, the number of terms doubles since every term will give two new terms after using Thm.\ \ref{Spectrally Positive One}. Since we go down $l-1$ rows, those $2j-1$ nodes will produce in total $(2j-1)2^{l-1}=2^lj-2^{l-1}$ nodes. Hence, we see that the element $\exp[{-(\alpha_1+\ldots+\alpha_{l-1}+\psi(q_l))x}]$ in the $n$-th row is at the position $2^lj - 2^{l-1}+1$. The numbering of the coefficients is based on this ordering.

% Set the overall layout of the tree
\tikzstyle{level 1}=[level distance=1.5cm, sibling distance=8cm]
\tikzstyle{level 2}=[level distance=1.5cm, sibling distance=4cm]
\tikzstyle{level 3}=[level distance=1.5cm, sibling distance=2cm]

% Define styles for bags and leafs
\tikzstyle{bag} = [text width=5.5em, text centered]
\tikzstyle{end} = [circle, minimum width=7pt,fill, inner sep=0pt]

\begin{figure}[H]
\begin{tikzpicture}[grow=down, sloped]
\node[bag] {}
    child {
        node[bag] {$L_1^{(1)}$}
            child {
                node[bag] {$L_1^{(2)}$}
                     child{
                        node[bag] {$L_1^{(3)}$}
                        edge from parent
                        node[above] {$$}
                        node[below]  {$$}
                     }
                     child{
                       node[bag] {$L_{(2,1)}^{(3)}$}
                        edge from parent
                        node[above] {$$}
                        node[below]  {$$}
                     }
                    edge from parent
                    node[above] {$$}
                    node[below]  {$$}
            }
            child {
                node[bag] {$L_{(2,1)}^{(2)}$}
                    child{
                        node[bag] {$L_{(3,2)}^{(3)}$}
                        edge from parent
                        node[above] {$$}
                        node[below]  {$$}
                     }
                     child{
                       node[bag] {$L_{(4,1)}^{(3)}$}
                         edge from parent
                         node[above] {$$}
                         node[below]  {$$}
                     }
                edge from parent
                node[above] {$$}
                node[below]  {$$}
            }
            edge from parent
            node[above] {}
            node[below]  {$$}
    }
    child {
        node[bag] {$L_{(2,1)}^{(1)}$}
            child {
                node[bag] {$L_{(3,2)}^{(2)}$}
                     child{
                        node[bag] {$L_{(5,3)}^{(3)}$}
                        edge from parent
                        node[above] {$$}
                        node[below]  {$$}
                     }
                     child{
                       node[bag] {$L_{(6,1)}^{(3)}$}
                        edge from parent
                        node[above] {$$}
                        node[below]  {$$}
                     }
                edge from parent
                node[above] {$$}
                node[below]  {$$}
            }
            child {
                node[bag] {$L_{(4,1)}^{(2)}$}
                  child{
                        node[bag] {$L_{(7,2)}^{(3)}$}
                        edge from parent
                node[above] {$$}
                node[below]  {$$}
                     }
                  child{
                       node[bag] {$L_{(8,1)}^{(3)}$}
                     edge from parent
                node[above] {$$}
                node[below]  {$$}
                     }
                edge from parent
                node[above] {$$}
                node[below]  {$$}
            }
            edge from parent
            node[above] {}
            node[below]  {$$}
    };

\end{tikzpicture}
\caption{The coefficients of the exponential terms}
\label{tree graph 2}
\end{figure}
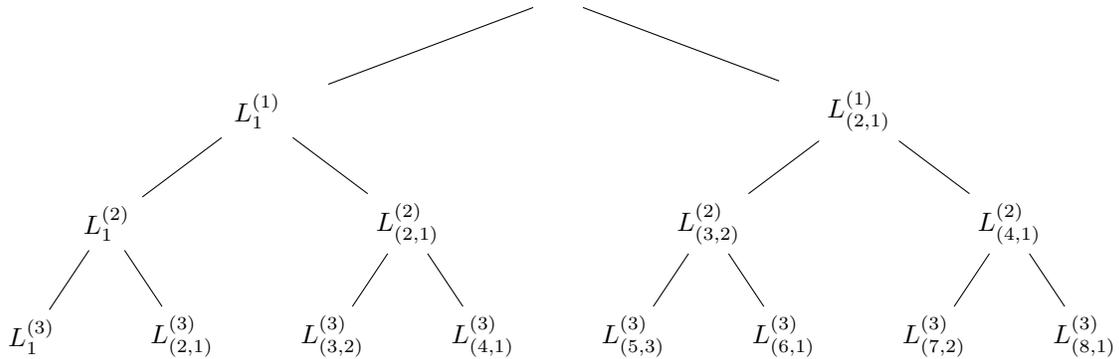

In Fig.\ \ref{tree graph 2}, we use the following notation:
\begin{itemize}
\item[$\circ$]
$L_1^{(n)}$ denotes the coefficient of the term $\exp[{-(\alpha_1+\ldots+\alpha_n)x}]$;
\item[$\circ$]
$L_{(2j,1)}^{(n)}$ denote the coefficients of $\exp[{-\psi(q_1)x}]$ (where $j=1,2,\ldots,2^{n-1}$);
\item[$\circ$]
$L_{(2^lj-2^{l-1}+1,l)}^{(n)}$  denote the coefficients of  $\exp[{-(\alpha_1+\ldots+\alpha_{l-1}+\psi(q_l))x}]$ (where $l=2,3,\ldots,n$ and $j=1,2,\ldots,2^{n-l}$). 
\end{itemize}We note here that the superscript $(n)$ in these factors corresponds to the number of exponential random variables considered (or, equivalently, in which row of the tree we are).   We now proceed to the main result for the case of a spectrally positive input process.

\begin{theorem}
\label{MainTheorem}
Suppose we have $n$ independent exponentially distributed random variables $T_1,\ldots,T_n$ with distinct parameters $q_1,...,q_n$. Then, for $\alpha_i\geq0$ and $x\geq0$, we have 
\[\E_xe^{-\alpha_1\Q_{T_1}-\alpha_2\Q_{T_1+T_2}-\ldots-\alpha_n\Q_{T_1+T_2+\ldots+T_n}}=\prod_{i=1}^n\frac{q_i}{q_i-\phi(\alpha_i+\ldots+\alpha_n)}e^{-(\alpha_1+\ldots+\alpha_n)x}\]
\begin{equation}
\label{superposition}
\hspace{4cm}+\,\sum_{l=1}^n\sum_{j=1}^{2^{n-l}}L_{(2^{l}j-2^{l-1}+1,l)}^{(n)}(\bar{q},\bar{\alpha})e^{-(\alpha_1+...+\alpha_{l-1}+\psi(q_l))x},
\end{equation}
where the coefficients $L_{(2^{l}j-2^{l-1}+1,l)}^{(n)}$ are defined below in Definition \ref{DefinitionL}.
\end{theorem}

The vectors $\bar{q}=(q_1,\ldots,q_n)$ and $\bar{\alpha}=(\alpha_1,\ldots,\alpha_n)$ are here explicitly included, so as to show the dependence of the coefficients on the $q$'s and $\alpha$'s. Later on these vectors are omitted to keep the notation concise.

\begin{definition}
\label{DefinitionL}
For $l=1,...,n$ and $j=1,\ldots,2^{n-l}$ we have
\[L_{(2^lj-2^{l-1}+1,l)}^{(n)}(\bar{q},\bar{\alpha})=c^{(2^lj-2^{l-1}+1,n)}\prod_{i=1}^n\frac{q_i}{q_i-\phi(\alpha_i+d^{(i,2^lj-2^{l-1}+1)})}\prod_{i=l}^n\frac{\alpha_i+d^{(i,2^lj-2^{l-1}+1)}}{d^{(i-1,2^lj-2^{l-1}+1)}}.\]
where the $c^{(2^lj-2^{l-1}+1,n)}$ are given below in Lemma \ref{lemmasign}, $d^{(n,2^lj-2^{l-1}+1)}=0$ and the $d^{(i,2^lj-2^{l-1}+1)}$, for $i=1,2,\ldots,n-1$, are given through
\[d^{(i,2^lj-2^{l-1}+1)}=\left \{ \begin{array}{ll}
\alpha_{i+1}+d^{(i+1,2^lj-2^{l-1}+1)} &\mbox{for } \ceil*{\frac{2^lj-2^{l-1}+1}{2^i}} \text{ odd,}\\
\\
\psi(q_{i+1}) &\mbox{for } \ceil*{\frac{2^lj-2^{l-1}+1}{2^i}} \text{ even.}
\end{array} \right.\]
\end{definition}

\begin{remark} {\em 
The terms $d^{(i,j)}$ are given from a recursive formula. The fact that this recursion is well defined, follows because the last term equals zero (i.e., $d^{(n,j)} = 0$ for all $j$'s).}
\end{remark}
%The terms $c^{(2^lj-2^{l-1}+1,n)}$ are given by the following lemma.

\begin{lemma}
\label{lemmasign}
Consider $j=1,2,...,2^n$ and take the binary representation of $j-1$, i.e., $j-1=b_0 2^0+b_12^1+\ldots+b_{n-1}2^{n-1}$. Then, for $c^{(j,n)}$ (or, equivalently, the sign of the $j$-th element in the $n$-th row of the tree presented above) we have 
\[c^{(j,n)}=(-1)^{\PAR\{b_0,\ldots,b_{n-1}\}},\]
where $\PAR\{b_0,\ldots,b_{n-1}\}$ is $0$ if the number of 1's in the binary expansion of $j-1$ is even and 1 if it is odd.
\end{lemma}

Similar to the Erlang-2 situation (i.e., $n=2$ and $q_1=q_2$), the case in which some of the $q_i$'s are the same has to be treated separately. For instance, $n$ successive applications of `l'H\^{o}pital' lead to an expression for $\E_x e^{-\alpha \Q_T}$ , when $T$ has an Erlang-$n$ distribution.

\subsection{Spectrally negative case}\label{section spectrally negative}

In this subsection we concentrate on the case of a spectrally negative input process $X$. The joint workload density has a structure that is very similar to that observed for the LST in the spectrally positive case. Due to the strong Markov property the joint density can be decomposed into
\begin{equation*}
\Pb_x(\Q_{T_1} \in {\rm d}y_1 ; \; \cdots \; ; \Q_{T_1+...+T_n}\in{\rm d}y_n) 
 = \Pb_x(\Q_{T_1} \in {\rm d}y_1) \; \cdots \; \Pb_{y_{n-1}}(\Q_{T_n} \in {\rm d}y_n). 
\end{equation*}
That is, the joint density is simply the product of densities at single exponential epochs, as given in Thm.~\ref{spectrally negative one exponential time}. Henceforth, we focus on the density of the workload process at consecutive exponential epochs, i.e., 
\begin{equation}
\label{density at n epochs}
\Pb_x(\Q_{T_1+...+T_n}\in{\rm d}y), \qquad y\geq0.
\end{equation}

First we illustrate how to obtain an expression for the density $\Pb_x(\Q_{T_1+T_2}\in{\rm d}y)$ for some initial workload $x\geq0$ and $y>0$. From Thm.\ \ref{spectrally negative one exponential time} we have an expression for the density $\Pb_x(\Q_T\in{\rm d}y)$. Consider now two exponentially distributed random variables $T_1,T_2$ with distinct parameters $q_1,q_2$. 
Conditioning on $\Q_{T_1}$ and applying Thm.\ \ref{spectrally negative one exponential time} twice yields
\begin{align} \label{eqn:pos-2}
\Pb_x(\Q_{T_1+T_2}\in{\rm d}y)=&\int_{z=0}^\infty \Pb_z(\Q_{T_2}\in{\rm d}y)\Pb_x(\Q_{T_1}\in{\rm d}z) \nonumber\\
=&\int_0^\infty \left( -q_1W^{(q_1)}(x-z)+\Psi(q_1)e^{-\Psi(q_1)z}Z^{(q_1)}(x)\right)\Pb_z(\Q_{T_2}\in{\rm d}y){\rm d}z \nonumber\\
=&\Bigg[q_1q_2\int_0^\infty W^{(q_1)}(x-z)W^{(q_2)}(z-y){\rm d}z \nonumber\\
&-q_1\Psi(q_2)e^{-\Psi(q_2)y}\int_0^\infty W^{(q_1)}(x-z)Z^{(q_2)}(z){\rm d}z \nonumber\\
&-q_2\Psi(q_1)Z^{(q_1)}(x)\int_0^\infty e^{-\Psi(q_1)z}W^{(q_2)}(z-y){\rm d}z \nonumber\\
&+ \Psi(q_1)\Psi(q_2)e^{-\Psi(q_2)y}Z^{(q_1)}(x)\int_0^\infty e^{-\Psi(q_1)z}Z^{(q_2)}(z){\rm d}z\Bigg]{\rm d}y.
\end{align}
After some standard calculus and using the definition of the $q$-scale functions we find the expression 
\begin{align*}
\Pb_x(Q_{T_1+T_2}\in\mathrm{d}y)=&\Bigg[q_1q_2\left( W^{(q_2)}\star W^{(q_1)}\right)(x-y)-\Psi(q_2)q_1e^{-\Psi(q_2)y}\left(Z^{(q_2)}\star W^{(q_1)}\right)(x)\\
&-\Psi(q_1)e^{-\Psi(q_1)y}\frac{q_2}{q_1-q_2}Z^{(q_1)}(x)+\Psi(q_2)e^{-\Psi(q_2)y}\frac{q_1}{q_1-q_2}Z^{(q_1)}(x)\Bigg]{\rm d}y. \numberthis \label{expression for density spectrally negative two times final}
\end{align*}

Again the case $q_1=q_2$ has to treated separately, by using l'H\^{o}pital's rule; the detailed computations corresponding to this case can be found in \cite[Section 4.2]{Man}. 

We see that by conditioning on the value of the workload at the first exponential epoch we can derive
the transform at two exponential epochs. As a next step our aim is to find an expression for (\ref{density at n epochs}) for an arbitrary $n>0$ and for exponentially distributed random variables $T_i$ with parameter $q_i$ ($i=1,\ldots,n$). Conditioning on the workload at the first $n-1$ exponential epochs yields
\begin{equation}
\label{transition spectrally negative}
\Pb_x(\Q_{T_1+\ldots+T_n}\in{\rm d}y)=\int_{z=0}^\infty \Pb_x(\Q_{T_1+\ldots+T_{n-1}}\in{\rm d}z)\Pb_z(Q_{T_n}\in{\rm d}y).
\end{equation}
For the case of a spectrally positive input process (which was the topic of the previous subsection) one should  condition on the value at the first exponential epoch, which allows the use of the induction hypothesis, but one needs to adjust the indices appropriately as the first exponential random variable is actually $T_2$. 
For the spectrally negative case, however, conditioning on the value of $T_1+...+T_{n-1}$ (and not only on $T_1$) allows us to circumvent this technicality. 

Moreover, the transition from step $n-1$ to $n$ can again be represented by using an elegant tree structure that is similar to the one developed for the spectrally positive case.
The expression for the density at $n-1$ exponential epochs has $2^{n-1}$ terms and each term produces two new terms when integrated with the density $\Pb_z(Q_{T_n}\in{\rm d}y)$ (with respect to $z$). We also notice that in the expression for $n$ exponentially distributed random variables the first term is always of the form $\left(W^{(q_n)}\star\ldots\star W^{(q_1)}\right)(x-y)$ while the other terms are of the form $\left(Z^{(q_l)}\star W^{(q_{l-1})}\star\ldots\star W^{(q_1)}\right)(x)$, for $l=1,2,\ldots,n$, multiplied by some coefficients that in general are functions of $y$. The underlying mechanism is illustrated in Fig.~\ref{tree spectrally negative}. In this tree the node $W^{(q_n)}(x-y)$ denotes the term $\left(W^{(q_n)}\star\ldots\star W^{(q_1)}\right)(x-y)$ while the nodes $Z^{(q_l)}(x)$, for $l=1,\ldots,n$, denote the terms $\left(Z^{(q_l)}\star W^{(q_{l-1})}\star\ldots\star W^{(q_1)}\right)(x)$. We see that at every row, say row $k$ for ease, a new subtree with root $\left( Z^{(q_k)}\star W^{(q_{k-1})}\star\ldots\star W^{(q_1)}\right)(x)$ is created. These terms do not change as we move downwards in the tree since they only depend on the initial workload $x$ and do not take part in the integrations, similar to those carried out in (\ref{eqn:pos-2}). Their coefficients change though, by a mechanism that is identified in the proof of our result.

% Set the overall layout of the tree
\tikzstyle{level 1}=[level distance=1.5cm, sibling distance=8cm]
\tikzstyle{level 2}=[level distance=1.5cm, sibling distance=4cm]
\tikzstyle{level 3}=[level distance=1.5cm, sibling distance=2cm]
\tikzstyle{level 4}=[level distance=1.5cm, sibling distance=2cm]

% Define styles for bags and leafs
\tikzstyle{bag} = [text width=5.5em, text centered]
\tikzstyle{end} = [circle, minimum width=5pt, maximum width 7pt, fill, inner sep=2pt]

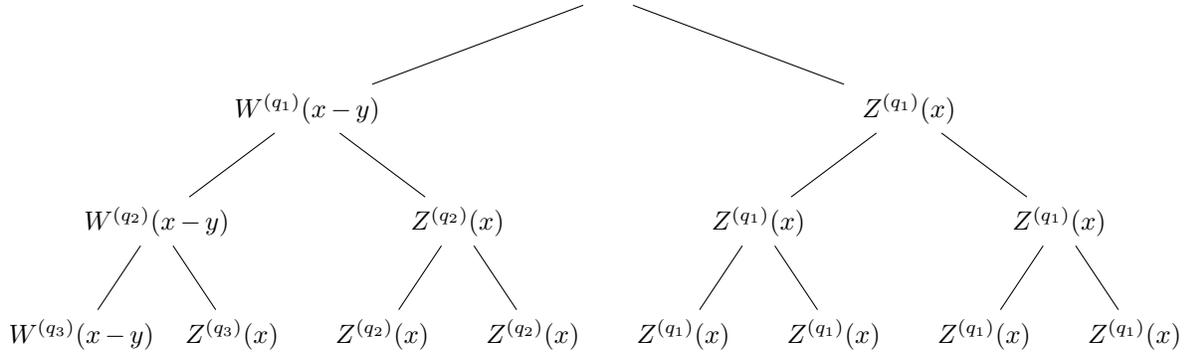
\begin{figure}[H]
\begin{tikzpicture}[grow=down, sloped]
\node[bag] {}
    child {
        node[bag] {$W^{(q_1)}(x-y)$}
            child {
                node[bag] {$W^{(q_2)}(x-y)$}
                     child{
                        node[bag] {$W^{(q_3)}(x-y)$}
                           {}
                        edge from parent
                node[above] {$$}
                node[below]  {$$}
                     }
                     child{
                       node[bag] {$ Z^{(q_3)}(x)$}
                     edge from parent
                node[above] {$$}
                node[below]  {$$}
                     }
                edge from parent
                node[above] {$$}
                node[below]  {$$}
            }
            child {
                node[bag] {$Z^{(q_2)}(x)$}
                    child{
                        node[bag] {$Z^{(q_2)}(x)$}
                        edge from parent
                        node[above] {$$}
                        node[below]  {$$}
                     }
                     child{
                       node[bag] {$Z^{(q_2)}(x)$}
                     edge from parent
                node[above] {$$}
                node[below]  {$$}
                     }
                edge from parent
                node[above] {$$}
                node[below]  {$$}
            }
            edge from parent
            node[above] {}
            node[below]  {$$}
    }
    child {
        node[bag] {$Z^{(q_1)}(x)$}
            child {
                node[bag] {$Z^{(q_1)}(x)$}
                     child{
                        node[bag] {$Z^{(q_1)}(x)$}
                          {}
                        edge from parent
                        node[above] {$$}
                        node[below]  {$$}
                     }
                     child{
                       node[bag] {$Z^{(q_1)}(x)$}
                     edge from parent
                node[above] {$$}
                node[below]  {$$}
                     }
                edge from parent
                node[above] {$$}
                node[below]  {$$}
            }
            child {
                node[bag] {$Z^{(q_1)}(x)$}
                  child{
                        node[bag] {$Z^{(q_1)}(x)$}
                        edge from parent
                node[above] {$$}
                node[below]  {$$}
                     }
                  child{
                       node[bag] {$Z^{(q_1)}(x)$}
                     edge from parent
                node[above] {$$}
                node[below]  {$$}
                     }
                edge from parent
                node[above] {$$}
                node[below]  {$$}
            }
            edge from parent
            node[above] {}
            node[below]  {$$}
    };

\end{tikzpicture}
\caption{The convolution terms at every step}
\label{tree spectrally negative}
\end{figure}

\noindent
We now proceed with the main result for the spectrally negative case.
\begin{theorem}
\label{spectrally negative general case density}
Suppose we have $n$ independent exponentially distributed random variables $T_1,\ldots,T_n$ with distinct parameters $q_1,\ldots,q_n$. The density of $\Q_{T_1+...+T_n}$, given that $\Q_0=x$, is given by

\begin{align*}
\Pb_x(\Q_{T_1+\ldots+T_n}\in\mathrm{d}y) =& \Bigg[ (-1)^n \prod_{i=1}^nq_i \cdot \left(W^{(q_n)}\star\ldots\star W^{(q_1)}\right)(x-y)\\
& +\sum_{l=1}^n\sum_{j=1}^{2^{n-l}}L_{(2^lj-2^{l-1}+1,l)}^{(n)}(y)\left( Z^{(q_l)}\star W^{(q_{l-1})}\star\ldots\star W^{(q_1)}\right)(x)\Bigg]\mathrm{d}y,
\end{align*}
where the coefficients $L_{(2^lj-2^{l-1}+1,l)}^{(n)}(y)$ %for $l=1,...,n$ and $j=1,...,2^{n-l}$
 are given in Definition \ref{lambdaexpo}.
\end{theorem}
\begin{definition}
\label{lambdaexpo}
For $l=1,\ldots,n$ and $j=1,\ldots,2^{n-l}$, we have the following expression
\[L_{(2^lj-2^{l-1}+1,l)}^{(n)}(y) = c^{(2^lj-2^{l-1}+1,n)}\Psi(q_{m(j,l)})e^{-\Psi(q_{m(j,l)})y}\prod_{\mathclap{\substack{i=1,\\ i\neq m(j,l)}}}^nq_i\prod_{i=l}^{m(j,l)}\frac{1}{q_{i}-q_{i+1}}\prod_{i=m(j,l)+1}^{n-1}\frac{1}{q_{m(j,l)}-q_{i+1}},\]
where $m(j,l) = \min\{ k\in\N: \ceil{\frac{2^lj-2^{l-1}+1}{2^{k}}}=1\}$. 
The terms $c^{(2^lj-2^{l-1}+1,n)}$ are given below in Lemma \ref{lemma sign 2}.  
\end{definition}

\begin{lemma}
\label{lemma sign 2}
Consider $j=1,2,\ldots,2^n$ and take the binary representation of $2^n-j$, $2^n-j=\beta_0\cdot2^0+\ldots+\beta_{n-1}\cdot2^{n-1}$. Then, for $c^{(j,n)}$ (or, equivalently, the sign of the $j$-th element in the $n$-th row of the tree presented above) we have the following formula
\[c^{(j,n)}=(-1)^{\PAR\{\beta_0,\beta_1,\ldots,\beta_{n-1}\}},\]
where $\PAR\{\beta_0,\ldots,\beta_{n-1}\}$ is $0$ if the number of 1's in the binary expansion of $2^n-j$ is even and 1 if it is odd.
\end{lemma}
\vspace{2mm}
Using the result obtained in Thm.\ \ref{spectrally negative general case density} we can find an expression for the transform with respect to the initial workload as well; again analytic continuation is used to obtain the result for any $\beta>0$.  
\begin{corollary}
\label{general triple transform}
For $\alpha>0$, ${\beta>0}$ and for $n$ independent exponentially distributed random variables $T_1,\ldots,T_n$ with distinct parameters $q_1,\ldots,q_n$, we have\[\int_0^\infty e^{-\beta x} \E_x e^{-\alpha \Q_{T_1+\ldots+T_n}}\mathrm{d}x  = c^{(1,n)}\prod_{i=1}^nq_i\frac{1}{\alpha+\beta}\prod_{i=1}^n\frac{1}{\Phi(\beta)-q_i}\]
\[+\sum_{l=1}^n\sum_{j=1}^{2^{n-l}}c^{(2^lj-2^{l-1}+1,n)}\prod_{\mathclap{\substack{i=1,\\ i\neq m(j,l)}}}^n q_i\hspace{-1mm}\prod_{i=l}^{m(j,l)}\hspace{-1mm} \frac{1}{q_{i}-q_{i+1}}\prod_{i=m(j,l)+1}^{n-1}\hspace{-1mm}\frac{1}{q_{m(j,l)}-q_{i+1}}\cdot \frac{\Psi(q_{m(j,l)})}{\alpha+\Psi(q_{m(j,l)})}\prod_{i=1}^{l}\frac{1}{\Phi(\beta)-q_i}\frac{\Phi(\beta)}{\beta},\]
where $m(j,l)$ and $c^{(2^lj-2^{l-1}+1,n)}$ are given in Definition \ref{lambdaexpo} and Lemma~\ref{lemma sign 2}.
\end{corollary}

The case in which some of the $q_i$'s are the same should be treated separately again. For instance, the density of $\Pb_x(\Q_T\in{\rm d}y)$ for $T$ having an Erlang-$n$ distribution follows after $n$ applications of l'H\^{o}pital's rule.

\section{Numerical Calculations}

In this section we present numerical illustrations of the transient workload behavior. We consider examples corresponding to the spectrally positive case (noting that the spectrally negative case can be dealt with similarly). The expression found in Thm.\ \ref{MainTheorem} is, from an algorithmic standpoint, highly attractive; the only drawback is that for every $n$ we have to compute $2^n$ terms, thus increasing the computation time significantly at every step.
In our illustrations, we consider the impact of $n$, i.e., the number of exponential variables. 
We also comment on ways to determine the workload distribution  at a fixed (deterministic, that is) time; the mean idea there, as we point out in more detail below, is to approximate a deterministic epoch $t$ by the sum of exponentially distributed random variables with appropriately chosen parameters. 

%An other technicality that has to be taken under consideration is the choice of the parameters $q_1,...,q_n$. We saw that the special case $q_1=...=q_n$ needs a special analysis; as we will see below there is a simple way to overcome this difficulty without having to apply l'H\^{o}pital's rule in the expression found in Theorem \ref{MainTheorem}. 

We focus on two specific L\'evy processes: Brownian motion and the Gamma process. For the case of Brownian motion, the input process $X$  is a Brownian motion with a drift, henceforth denoted by $X\in\BM(d,\sigma^2)$. Then, $\phi(\alpha):=\log\E e^{-\alpha X_1}=-\alpha d+\alpha^2\sigma^2 / 2$, and the right inverse function is 
\[\psi(q)=\frac{d+\sqrt{d^2+2\sigma^2q}}{\sigma^2}.\]
For {\em reflected Brownian motion} (i.e., the workload of a queue with Brownian motion as input) there is an explicit expression for the conditional distribution $\Pb(\Q_t\leq y|\Q_0=x)$, for $y>0$, see e.g.\ \cite[Section 1.6]{Har}. It is a matter of straightforward calculus to use this formula to find an expression for the transform
\[\E_xe^{-\alpha\Q_t}=\int_0^\infty e^{-\alpha y}\Pb(\Q_t\in{\rm d} y|\Q_0=x).\]
This result is used to evaluate the performance of our procedure in case a fixed time $t$ is approximated by the sum of exponentials.

The {\em Gamma} process is characterized by the L\'evy-Khintchine triplet $(d,\sigma^2,\Pi)$, where $\sigma^2=0$; the L\'evy measure is given, for some $\beta,\gamma>0$, by $\Pi({\rm d}x)=(\beta/x) \, e^{-\gamma x}$, for $x>0$, and the drift is $d=\int_0^1x\Pi({\rm d}x)$. From the definition of the L\'evy measure we see that the Gamma process is a spectrally positive process with a.s. non-decreasing sample paths. We also add  a negative drift  such that the Laplace transform is equal to 
\[\phi(\alpha) :=\log\E e^{-\alpha X_1} = \beta\log\left(\frac{\gamma}{\gamma+\alpha}\right)+\rho\alpha,\]
where $\rho> \beta / \gamma$ in case of a negative drift $d=(\beta/\gamma)-\rho$. For the Gamma process with parameters $\gamma,\beta>0$ and a drift $(\beta/\gamma)-\rho$ we use the notation $\mathbb{G}(\gamma,\beta,\rho)$. If the input is a Gamma process there is no explicit expression for the transform $\E_xe^{-\alpha\Q_t}$ in contrast with the case of a Brownian input.

Suppose now that we wish to characterize the distribution of $\Q_t$ for a deterministic $t$.
The idea is that we can approximate $t$ by a sum of, say $n$, independent exponential random variables. 
An optimal choice of the parameters $q_i$ then follows from solving the following constrained optimization problem:
\[\min\VAR(T_1+...+T_n) =\min\sum_{i=1}^n\frac{1}{q_i^2} \qquad \text{s.t.} \;\; \sum_{i=1}^n \E T_i=t.\]

This constrained optimization problem has solution $q_1=\ldots=q_n= n/t$. A complicating factor is that in Thm.\ \ref{MainTheorem} the parameters $q_i$ should be chosen distinct. To remedy this, we propose to impose a small perturbation of the optimal $q_i$'s such that they are distinct:
\begin{equation}
\label{parameter choice}
\frac{1}{q_i}=\frac{t}{n}(1+\alpha_i),
\end{equation}
where the $\alpha_i$ are suitably chosen {\em small} numbers that sum up to 0. In the two tables that follow we present the numerical results obtained from calculating the expression in Thm.\  \ref{MainTheorem} for the case $X\in\BM(-1,1)$ (Table \ref{Table for Brownian Motion}) and for the case $X\in\mathbb{G}(1,1,2)$ (Table \ref{Table for Gamma process}). Here, we consider the situation of $x=0$ and $t=1$. The parameters $q_i$ are chosen according to (\ref{parameter choice}) with, if $n$ is even, the $\alpha_i$'s given by
\[\alpha_i=\left\{\begin{array}{ll}
\hspace{2mm}0.01\cdot i &\mbox{ if } i=1,\ldots,\frac{n}{2}\\[1mm]
-0.01\cdot i &\mbox{ if } i=\frac{n}{2}+1,\ldots,n.
\end{array}\right.\]
(If $n$ is odd we choose $\alpha_{n+1/2}=0$ and the rest as indicated above.) In the first table, $X\in\BM(-1,1)$, we take $n=1,4,6,7,8$ and compare our approximations with the exact values obtained from $\E_x e^{-\alpha\Q_t}$ for different values of $\alpha$. In the last column we present the relative errors between the exact value and the approximation value for $n=8$. 

\begin{table}[t]
\caption{Numerical approximations for $X\in\BM(-1,1)$, $x=0$ and $t=1$} \label{Table for Brownian Motion}
  \begin{center}
      \begin{tabular}{| l | c | c | c | c | c | c | c | c | c|}

$$	&	$n=1$	&	$n=4$	&	$n=6$	&	$n=7$	&	$n=8$	&	exact value	&	relative error		\\ \hline
$\alpha=0.1$	&	0,9647	&	0,96064	&	0,96021	&	0,96005	&	0,96001	&	0,95914	&	-0,09	\%	\\ \hline
$\alpha=0.2$	&	0,9318	&	0,92410	&	0,92327	&	0,92299	&	0,92300	&	0,92128	&	-0,19	\%	\\ \hline
$\alpha=0.3$	&	0,9011	&	0,89008	&	0,88892	&	0,88851	&	0,88836	&	0,88611	&	-0,25	\% \\ \hline
$\alpha=0.4$	&	0,8723	&	0,85836	&	0,85688	&	0,85638	&	0,85608	&	0,85338	&	-0,32 \%		\\ \hline
$\alpha=0.5$	&	0,8453	&	0,82870	&	0,82696	&	0,82637	&	0,82590	&	0,82285	&	-0,37	\%	\\ \hline
$\alpha=0.6$	&	0,8199	&	0,80094	&	0,79896	&	0,79828	&	0,79786	&	0,79432	&	-0,44	\%	\\ \hline
$\alpha=0.7$	&	0,7960	&	0,77488	&	0,77270	&	0,77196	&	0,77237	&	0,76760	&	-0,62	\%	\\ \hline
$\alpha=0.8$	&	0,7735	&	0,75040	&	0,74803	&	0,74723	&	0,74625	&	0,74254	&	-0,50	\%	\\ \hline
$\alpha=0.9$	&	0,7522	&	0,72735	&	0,72482	&	0,72397	&	0,72415	&	0,71900	&	-0,71	\% \\ \hline
$\alpha=1$	&	0,7321	&	0,70562	&	0,70295	&	0,70205	&	0,70205	&	0,69684	&	-0,74	\%	\\ \hline

    \end{tabular}
\end{center}
\end{table}

\begin{table}[t]
\caption{Numerical approximations for $X\in\mathbb{G}(1,1,2)$, $x=0$ and $t=1$} \label{Table for Gamma process}
  \begin{center}
      \begin{tabular}{| l | l | l | l | l | l | l | l | l | l|}

$$	&	$n=1$	&	$n=4$	&	$n=5$	&	$n=6$	&	$n=7$	&	$n=8$	\\ \hline
$\alpha=0.1$	&	0,97582	&	0,99046	&	0,99037	&	0,99032	&	0,99028	&	0,99026	\\ \hline
$\alpha=0.2$	&	0,95527	&	0,98148	&	0,98130	&	0,98121	&	0,98112	&	0,98108	\\ \hline
$\alpha=0.3$	&	0,93754	&	0,97300	&	0,97275	&	0,97261	&	0,97249	&	0,97243	\\ \hline
$\alpha=0.4$	&	0,92205	&	0,96499	&	0,96465	&	0,96448	&	0,96432	&	0,96425	\\ \hline
$\alpha=0.5$	&	0,90838	&	0,95739	&	0,95699	&	0,95678	&	0,95659	&	0,95662	\\ \hline
$\alpha=0.6$	&	0,89621	&	0,95018	&	0,94972	&	0,94948	&	0,94925	&	0,94936	\\ \hline
$\alpha=0.7$	&	0,88530	&	0,94333	&	0,94281	&	0,94254	&	0,94228	&	0,94201	\\ \hline
$\alpha=0.8$	&	0,87543	&	0,93681	&	0,93623	&	0,93593	&	0,93565	&	0,93565	\\ \hline
$\alpha=0.9$	&	0,86647	&	0,93060	&	0,92996	&	0,92964	&	0,92933	&	0,92885	\\ \hline
$\alpha=1$	&	0,85828	&	0,92467	&	0,92398	&	0,92363	&	0,92330	&	0,92273	\\ \hline

    \end{tabular}
\end{center}
\end{table}

It should be realized that the numerical procedure has its limitations. First, from the expression in Thm.\ \ref{MainTheorem} we see that at every step we have to compute $2^n$ terms, which complicates the computation for $n$ large. We also see that when the parameters $q_i$ are `almost equal' i.e., $\alpha_i$ in (\ref{parameter choice}) is small) we add and  subtract terms that are large in absolute value (as the denominators featuring in the result of Thm.~\ref{MainTheorem} are close to zero), which potentially causes  instability. 
Our numerical tests show that the choice of the parameters $q_i$ influence the numerical stability; for the parameters indicated in (\ref{parameter choice}) the results begin to deviate for $n>9$ due to numerical issues. From Table \ref{Table for Brownian Motion}, corresponding to the case of a Brownian input process (for which we can compare with exact results), we see that for $n=8$ our relative error is below 1\%.
For the case of a Gamma input process, we verified that the transform converges to the steady-state workload as given by the {\em generalized Pollaczek-Khintchine} formula \cite[Thm.~3.2]{Man}.

As a second application of Thm.\ \ref{MainTheorem} we use the results obtained in Tables \ref{Table for Brownian Motion} and \ref{Table for Gamma process} to approximate the value of $\E_x\Q_t$ for the two cases $X\in\BM(-1,1)$ and $X\in\mathbb{G}(1,1,2)$, essentially relying on numerical differentiation. By considering an $\alpha$ sufficiently small and an $n$ sufficiently large we use the approximation 
\[\E_x\Q_t\sim\frac{1-\E_x e^{-\alpha\Q_{T_1+...+T_n}} }{\alpha}.\] 
We present our findings for the cases $X\in\BM(-1,1)$ and $X\in\mathbb{G}(1,1,2)$, respectively, displaying the qualitative behavior of $\E_x\Q_t$ as a function of time for various values of $x$. For the mean value of the stationary workload we know that $\E\Q={\phi''(0)}/({2\phi'(0)})$, as follows directly from the generalized Pollaczek-Khintchine formula. 

\begin{figure}[H]
\begin{centering}
\includegraphics[scale=0.8]{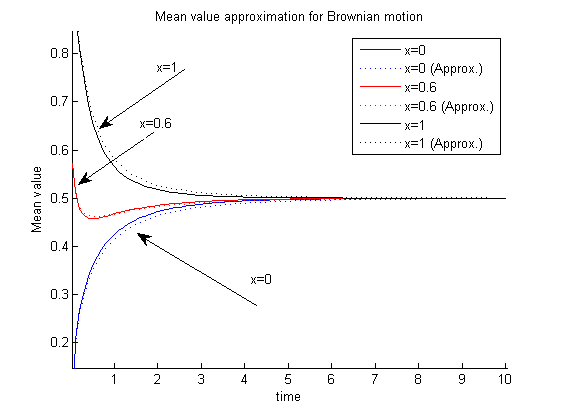}
\end{centering}
\caption{ Mean value approximation with $n=7$}
\label{BM graph}
\end{figure}

\begin{figure}[H]
\begin{centering}
\includegraphics[scale=0.8]{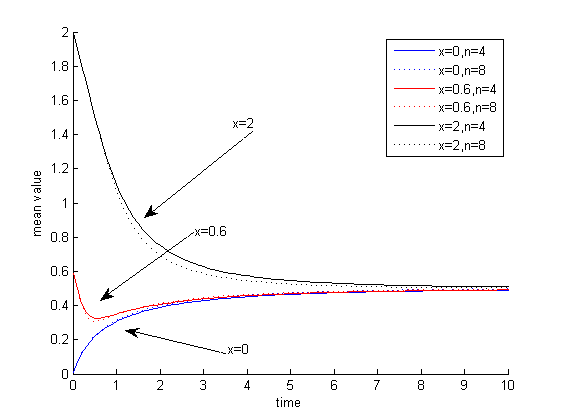}
\end{centering}
\caption{ Mean value for Gamma process}
\label{Gamma graph}
\end{figure}

In Figs.~\ref{BM graph} and \ref{Gamma graph} we observe three different scenarios corresponding to different values of the initial workload. When the initial workload is 0, the mean workload increases and converges to the mean value of the steady-state workload. This follows directly, as, for any L\'evy input process, $\Q_t\stackrel{\rm d}{=}\sup_{0\leq s\leq t}X_s$ when the initial workload is 0, implying that $\E_0\Q_t$ is increasing over time. When the initial workload is slightly above the steady-state workload, it is interesting to notice that $\E \Q_t$ first decreases below the steady-state version, and then converges from below. For higher initial workloads, $\E \Q_t$ is always decreasing and converges to the steady-state value from above.

\section{Proofs} \label{sec:proofs}

In this section we prove the main results. For both cases of spectrally one-sided processes, we first show an auxiliary lemma relating to the signs of each term. The main results are then proved using induction.

\subsection{Proof of Theorem \ref{MainTheorem}}

Before deriving the main result, we first  prove Lemma \ref{lemmasign}, which gives the sequence of the $2^n$ signs that appear in the expression of the transform at a time epoch corresponding to the sum of $n$ exponentially distributed random variables. 
From Thm.\ \ref{Spectrally Positive One} we see that for $n=1$ the signs of the coefficients are $+,-$. For $n=2$ and from Eqn.\ (\ref{case2}) we see that the signs are $+,-,-,+$ (where it is  noted that we use the ordering of the terms presented in Fig.\ \ref{tree graph 2}). Since we know how the terms are produced when we go from the step with $n$ exponential times to the step with $n+1$ exponential random variables (see Section~\ref{section spectrally positive}) we see that the signs at every step can be represented again by a tree graph. 
In this tree, row $n$ again consists of $2^n$ nodes and, starting from the left, the nodes represent the sign of every factor when the expression is written as in Eqn.\ (\ref{superposition}).

% Set the overall layout of the tree
\tikzstyle{level 1}=[level distance=1cm, sibling distance=7cm]
\tikzstyle{level 2}=[level distance=1cm, sibling distance=3cm]
\tikzstyle{level 3}=[level distance=1cm, sibling distance=1.5cm]
\tikzstyle{level 4}=[level distance=1cm, sibling distance=0.5cm]

% Define styles for bags and leafs
\tikzstyle{bag} = [text width=8em, text centered]
\tikzstyle{end} = [circle, minimum width=7pt,fill, inner sep=0pt]
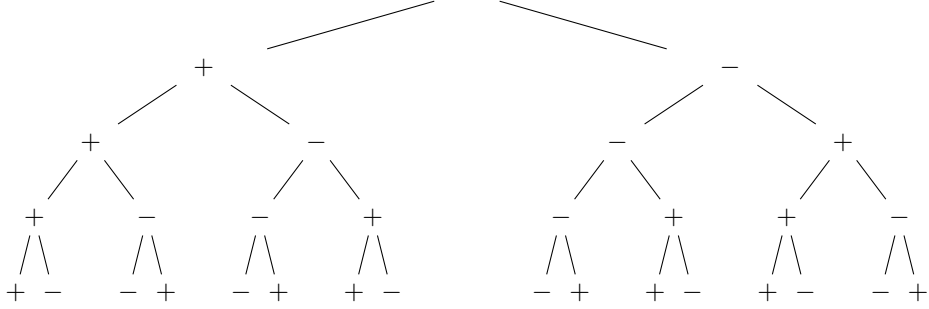
\begin{figure}[H]

\begin{tikzpicture}[grow=down, sloped]
\node[bag] {}
    child {
        node[bag] {$+$}
            child {
                node[bag] {$+$}
                     child{
                        node[bag] {$+$}
                           child{
                                node [bag] {$+$}
                                edge from parent
                                node[above] {$$}
                                node[below]  {$$}
                           }
                            child{
                                node [bag] {$-$}
                                edge from parent
                                node[above] {$$}
                                node[below]  {$$}
                           }
                        edge from parent
                        node[above] {$$}
                        node[below]  {$$}
                     }
                     child{
                       node[bag] {$-$}
                          child{
                                node [bag] {$-$}
                                edge from parent
                                node[above] {$$}
                                node[below]  {$$}
                           }
                            child{
                                node [bag] {$+$}
                                edge from parent
                                node[above] {$$}
                                node[below]  {$$}
                           }
                        edge from parent
                        node[above] {$$}
                        node[below]  {$$}
                     }
                    edge from parent
                    node[above] {$$}
                    node[below]  {$$}
            }
            child {
                node[bag] {$-$}
                    child{
                        node[bag] {$-$}
                            child{
                                node [bag] {$-$}
                                edge from parent
                                node[above] {$$}
                                node[below]  {$$}
                           }
                            child{
                                node [bag] {$+$}
                                edge from parent
                                node[above] {$$}
                                node[below]  {$$}
                           }
                        edge from parent
                        node[above] {$$}
                        node[below]  {$$}
                     }
                     child{
                       node[bag] {$+$}
                              child{
                                node [bag] {$+$}
                                edge from parent
                                node[above] {$$}
                                node[below]  {$$}
                           }
                            child{
                                node [bag] {$-$}
                                edge from parent
                                node[above] {$$}
                                node[below]  {$$}
                           }
                         edge from parent
                         node[above] {$$}
                         node[below]  {$$}
                     }
                edge from parent
                node[above] {$$}
                node[below]  {$$}
            }
            edge from parent
            node[above] {}
            node[below]  {$$}
    }
    child {
        node[bag] {$-$}
            child {
                node[bag] {$-$}
                     child{
                        node[bag] {$-$}
                            child{
                                node [bag] {$-$}
                                edge from parent
                                node[above] {$$}
                                node[below]  {$$}
                           }
                            child{
                                node [bag] {$+$}
                                edge from parent
                                node[above] {$$}
                                node[below]  {$$}
                           }
                        edge from parent
                        node[above] {$$}
                        node[below]  {$$}
                     }
                     child{
                       node[bag] {$+$}
                            child{
                                node [bag] {$+$}
                                edge from parent
                                node[above] {$$}
                                node[below]  {$$}
                           }
                            child{
                                node [bag] {$-$}
                                edge from parent
                                node[above] {$$}
                                node[below]  {$$}
                           }
                        edge from parent
                        node[above] {$$}
                        node[below]  {$$}
                     }
                edge from parent
                node[above] {$$}
                node[below]  {$$}
            }
            child {
                node[bag] {$+$}
                  child{
                        node[bag] {$+$}
                            child{
                                node [bag] {$+$}
                                edge from parent
                                node[above] {$$}
                                node[below]  {$$}
                           }
                            child{
                                node [bag] {$-$}
                                edge from parent
                                node[above] {$$}
                                node[below]  {$$}
                           }
                        edge from parent
                node[above] {$$}
                node[below]  {$$}
                     }
                  child{
                       node[bag] {$-$}
                          child{
                                node [bag] {$-$}
                                edge from parent
                                node[above] {$$}
                                node[below]  {$$}
                           }
                            child{
                                node [bag] {$+$}
                                edge from parent
                                node[above] {$$}
                                node[below]  {$$}
                           }
                     edge from parent
                node[above] {$$}
                node[below]  {$$}
                     }
                edge from parent
                node[above] {$$}
                node[below]  {$$}
            }
            edge from parent
            node[above] {}
            node[below]  {$$}
    };
\end{tikzpicture}
\caption{The sequence of the signs at every step}
\label{tree graph 3}

\end{figure}
We see that row $n+1$ can be derived from row $n$ when substituting every $+$ in row $n$ by the pair $+,-$, and every $-$ by the pair $-,+$. We can understand why this holds by looking at the expression in Thm.\ \ref{Spectrally Positive One} and the mechanism analyzed in Section \ref{section spectrally positive}. Denote by $c^{(j,n)}$ the sign of the $j$-th element in the $n$-th row in the above tree. Then $c^{(j,n)}$, for $j=1,2,\ldots,2^n$, corresponds to the sign of the $j$-th coefficient when considering $n$ exponentially distributed  in Eqn.\ (\ref{superposition}).

\begin{remark}\label{Samurai X Remark}{\em 
We observe that, because of symmetry, for the signs of the $k$-th row it holds that, for $j=1,\ldots,2^{k-1}$,
\[c^{(j,k)}=-c^{(j+2^{k-1},k)}.\]
 Hence the signs $j$ and $j+2^{k-1}$ in the $k$-th row will always be opposite.}
\end{remark}

\begin{proof}[Proof of Lemma \ref{lemmasign}]
We prove the lemma by induction on the number of exponentially distributed random variables. 

\begin{enumerate}
\item[(i)]
For $n=1$ we have two nodes and this case corresponds to the signs of the expression derived in Thm.\ \ref{Spectrally Positive One} for one exponentially distributed random variable $T$. We have that $c^{(1,1)}=+1$ and $c^{(2,1)}=-1$. Then we need the binary expansions of 0 and 1 which have no 1's and one 1, respectively. We see that $c^{(1,1)}=(-1)^0=1$ and $c^{(2,1)}=(-1)^1=-1$.
\item[(ii)]
We assume that the lemma holds for $n=k$. Hence, for $j=1,\ldots,2^k$, we have
\[c^{(j,k)}=(-1)^{\PAR\{b_0,\ldots,b_{k-1}\}}.\]
Here we make the following observation. In the tree presented above, consider an arbitrary row $n$. The $2^n$ signs of that row and the first $2^{n}$ signs of the $(n+1)$-th row are the same.

Now consider  the $(k+1)$-th row. Using the observation above and the induction hypothesis, the lemma holds for the first $2^k$ signs of this $(k+1)$-th row. Hence, we need to prove this statement only for $j=2^k+1,\ldots,2^{k+1}$.

For $j=1,2,\ldots,2^k$ we have
\[c^{(j,k+1)} = (-1)^{\PAR\{b_0,\ldots,b_{k-1}\}}= (-1)^{\PAR\{b_0,\ldots,b_{k-1},0\}}\]
where $j-1=b_0+b_1\cdot2+\ldots+b_{k-1}\cdot2^{k-1}+0\cdot2^k$. Consider now the element $j'=j+2^k$. From Remark~\ref{Samurai X Remark} we know that $c^{(j',k+1)}=-c^{(j,k+1)}$. We also know that the binary expansion of $j'$ has one more 1 than the binary expansion of $j$ since we add $2^k$, i.e., $j'-1=j-1+2^k=b_0+\ldots+b_{k-1}\cdot2^{k-1}+1\cdot2^k$, which shows that $(-1)^{\PAR\{b_0,\ldots,b_{k-1},1\}}=-(-1)^{\PAR\{b_0,\ldots,b_{k-1}\}}$ leading to
\[c^{(j,k+1)}=(-1)^{\PAR\{b_0,\ldots,b_k\}},\]
for all $j=1,2,\ldots,2^{k+1}$.
\end{enumerate}
\end{proof}

Before proceeding with the proof of Theorem \ref{MainTheorem} we present some general remarks which are used in the proofs of Thms.\ \ref{MainTheorem} and \ref{spectrally negative general case density}.

\begin{remark}{\em
\label{spirit lance remark}
For $l=2,\ldots,n$ and $j=1,\ldots,2^{n-l}$ we observe the following
\begin{enumerate}
\item[(a)] $2^lj-2^{l-1}+1$ is an odd number.
\item[(b)] For all $i=1,2,\ldots,l-2$,
\[\ceil*{\frac{2^lj-2^{l-1}+1}{2^i}}=\ceil*{2^{l-1-i}(2j-1)+\frac{1}{2^{i}}}=2^{l-i-1}(2j-1)+1,\]
which is always an odd number. In addition,
\[\ceil*{\frac{2^{l}j-2^{l-1}+1}{2^{l-1}}}=2j\] is an even number.
\item[(c)] 
\[\ceil*{\frac{2^lj-2^{l-1}+1}{2^l}}=\ceil*{j-\frac{1}{2}+\frac{1}{2^l}}=j\]
\item[(d)] For $i=0,1,\ldots$
\[\ceil*{\frac{2^lj-2^{l-1}+1}{2^{l+i}}}=\ceil*{\frac{j}{2^i}-\frac{1}{2^{i+1}}+\frac{1}{2^{l+i}}}=\ceil*{\frac{j}{2^{i}}}.\]
\end{enumerate}}
\end{remark}

\begin{proof}[Proof of Theorem \ref{MainTheorem}]
We use induction on the number of exponential random variables $T_1,...,T_n$. For the proof it is sufficient to start with $n=1$ (where it can be readily checked that Thm.\ \ref{MainTheorem} holds for $n=1$), but the case $n=2$ is more instructive. The joint transform for $n=2$ can be found in Eqn.\ (\ref{case2}). 

First of all, when $n=2$ we have in total $2^2=4$ terms. We see that the even terms correspond to $\exp[{-\psi(q_1)x}]$, and the third term corresponds to $\exp[{-(\alpha_1+\psi(q_2))x}]$.  According to (\ref{superposition}) the coefficient of $\exp[{-(\alpha_1+\alpha_2)x}]$ must be equal to
\[\frac{q_2}{q_2-\phi(\alpha_2)}\frac{q_1}{q_1-\phi(\alpha_1+\alpha_2)},\]
following directly from (\ref{case2}). 
We have two coefficients corresponding to $\exp[{-\psi(q_1)}]$, which according to (\ref{superposition}) should be equal to $L_{(2,1)}^{(2)}$ and $L_{(4,1)}^{(2)}$. Using Definition~\ref{DefinitionL}, we find the following expressions
\[L_{(2,1)}^{(2)}=-\prod_{i=1}^2\frac{q_i}{q_i-\phi(\alpha_i+d^{(i,2)})}
\prod_{i=1}^2 \frac{\alpha_i+d^{(i,2)}}{d^{(i-1,2)}}
=-\frac{q_2}{q_2-\phi(\alpha_2)}\frac{q_1}{q_1-\phi(\alpha_1+\alpha_2)}\frac{\alpha_1+\alpha_2}{\psi(q_1)},\]
as $d^{(0,2)}=\psi(q_1)$, $d^{(1,2)}=\alpha_2$, and $d^{(2,2)}=0$. Moreover,
\[L_{(4,1)}^{(2)}=\prod_{i=1}^2\frac{q_i}{q_i-\phi(\alpha_i+d^{(i,4)})} 
\prod_{i=1}^2 \frac{\alpha_i+d^{(i,4)}}{d^{(i-1,4)}},\]
where we see from the table for the factors $d^{(i,j)}$ (see Definition \ref{DefinitionL}) that $d^{(0,4)}=\psi(q_1)$, $d^{(1,4)}=\psi(q_2)$, and $d^{(2,4)}=0$. This leads to the following result
\[L_{(4,1)}^{(2)}=\frac{q_2}{q_2-\phi(\alpha_2)}\frac{\alpha_2}{\psi(q_2)}\frac{q_1}{q_1-\phi(\alpha_1+\psi(q_2))}\frac{\alpha_1+\psi(q_2)}{\psi(q_1)}.\]
For the last term, the coefficient of $e^{-(\alpha_1+\psi(q_2))x}$, we get
\[L_{(3,2)}^{(2)}=-\frac{q_1}{q_1-\phi(\alpha_1+\psi(q_2))}\frac{q_2}{q_2-\phi(\alpha_2)}
\prod_{i=2}^2 \frac{\alpha_i+d^{(i,3)}}{d^{(i-1,3)}}. \]
Since $d^{(1,3)}=\psi(q_2)$ and $d^{(2,3)}=0$, this agrees with Eqn.\ (\ref{superposition}), and thus the results holds for $n=2$.

\bigskip \noindent
We now assume that our formula holds for $n=k-1$. Hence we have that
\[ \E_xe^{-\alpha_1\Q_{T_1}-...-\alpha_{k-1}\Q_{T_1+\ldots+T_{k-1}}}=\prod_{i=1}^{k-1}\frac{q_i}{q_i-\phi(\alpha_{i}+\ldots+\alpha_{k-1})}e^{-(\alpha_1+\ldots+\alpha_{k-1})x}\]
\begin{equation}
\label{superposition2}
+\sum_{l=1}^{k-1}\sum_{j=1}^{2^{k-l-1}}L_{(2^{l}j-2^{l-1}+1,l)}^{(k-1)}e^{-(\alpha_1+\ldots+\alpha_{l-1}+\psi(q_l))x},
\end{equation}
where the coefficients $L_{(2^{l}j-2^{l-1}+1,l)}^{(k-1)}$ are given by Definition \ref{DefinitionL} for $n=k-1$ and the signs of all the factors are given by Lemma \ref{lemmasign}. In the induction step we prove this theorem for $n=k$ given that it holds for $n=k-1$. The expression for $n=k$ is derived from calculating the  integral
\[
\mathcal{L} := \E_xe^{-\alpha_1\Q_{T_1}-\ldots-\alpha_{k}\Q_{T_1+\ldots+T_{k}}} %\mathcal{I}
=\int_0^\infty e^{-\alpha_1y}\E_ye^{-\alpha_2\Q_{T_2}-...-\alpha_k\Q_{T_2+\ldots+T_k}}\Pb_x(\Q_{T_1}\in {\rm d}y),\]
where the expectation in the integral is known by the induction hypothesis. Here we see that we must raise all indices in (\ref{superposition2}) by one when we do the calculations because we start from time $T_2$ with parameter $q_2$ instead of from $T_1$. Combining the above with (\ref{superposition2}), we obtain
\begin{eqnarray}
%\E_xe^{-\alpha_1\Q_{T_1}-...-\alpha_{k}\Q_{T_1+...+T_{k}}} 
\mathcal{L} &=&\prod_{i=1}^{k-1}\frac{q_{i+1}}{q_{i+1}-\phi(\alpha_{i+1}+...+\alpha_k)}\int_0^\infty e^{-(\alpha_1+...+\alpha_k)y}\Pb_x(\Q_{T_1}\in{\rm d}y) \nonumber \\
 & &+\sum_{l=1}^{k-1}\sum_{j=1}^{2^{k-1-l}}L_{(2^lj-2^{l-1}+1,l)}^{(k-1)}\int_0^\infty e^{-(\alpha_1+...+\alpha_{l}+\psi(q_{l+1}))y}\Pb_x(\Q_{T_1}\in{\rm d}y) =: \mathcal{I} + \mathcal{II} \label{Hades equation}
\end{eqnarray}
The two integrals in (\ref{Hades equation}) can be computed using Thm.\ \ref{Spectrally Positive One}. Each integral gives two new terms, corresponding to a move down and left for the first term, and down and right for the second term in the trees presented in Figs.\ \ref{tree diagram 1} and \ref{tree graph 2}. The exponents are easily observed after an application of Thm.\ \ref{Spectrally Positive One}. Therefore, below we primarily focus on the coefficients. When considering such integrals the two terms obtained are referred to as the first and second term and are denoted by adding a 1 or 2 as indices to $\mathcal{I}$ and $\mathcal{II}$. We now successively consider (the coefficients of) $\mathcal{I}_1$,  $\mathcal{II}_1$,  $\mathcal{I}_2$, and  $\mathcal{II}_2$.

\bigskip \noindent
$\circ$ {\it Coefficient of $\mathcal{I}_1$.} \hspace{1mm}
The coefficient of $\exp[{-(\alpha_1+...+\alpha_k)x}]$ is found, using Thm.\ \ref{Spectrally Positive One}, from the first term of the integral
\[\int_0^\infty e^{-\alpha_1 y}\prod_{i=1}^{k-1}\frac{q_{i+1}}{q_{i+1}-\phi(\alpha_{i+1}+...+\alpha_{k})}e^{-(\alpha_2+...+\alpha_k)y}\Pb_x(\Q_{T_1}\in {\rm d}y),\]
which is
\begin{equation}
\label{Dias equation}
\prod_{i=2}^k\frac{q_i}{q_i-\phi(\alpha_i+...+\alpha_k)}\frac{q_1}{q_1-\phi(\alpha_1+...+\alpha_k)}=\prod_{i=1}^{k}\frac{q_{i}}{q_{i}-\phi(\alpha_i+...+\alpha_{k})}.
\end{equation}
This corresponds to the coefficient of the first term in Thm.\ \ref{MainTheorem}.

\bigskip \noindent
$\circ$ {\it Coefficient of $\mathcal{II}_1$.} \hspace{1mm}
 For $l=2,3,\ldots,k$ it is seen that the terms $L_{(2^lj-2^{l-1}+1,l)}^{(k)}$ for $j=1,2,\ldots,2^{k-l}$ can be derived from  the terms $L_{(2^{l-1}j-2^{l-2}+1,l-1)}^{(k-1)}$ by taking the first term of the integrals (this corresponds to a move down and left when we look at the tree in Fig.\ \ref{tree graph 2}):
\begin{equation}
\label{integraal}
\int_0^\infty L_{(2^{l-1}j-2^{l-2}+1,l-1)}^{(k-1)}e^{-\alpha_1y}e^{-(\alpha_2+\ldots+\alpha_{l-1}+\psi(q_l))y}\Pb_x(\Q_{T_1}\in { \rm d}y).
\end{equation}
From Thm.\  \ref{Spectrally Positive One} we obtain 
\begin{eqnarray*}
L_{(2^lj-2^{l-1}+1,l)}^{(k)} 
&=&L_{(2^{l-1}j-2^{l-2}+1,l-1)}^{(k-1)}\cdot\frac{q_1}{q_1-\phi(\alpha_1+\ldots+\alpha_{l-1}+\psi(q_l))}\\
&=&c^{(2^{l-1}j-2^{l-2}+1,k-1)}\cdot\prod_{i=1}^{k-1}\frac{q_{i+1}}{q_{i+1}-\phi(\alpha_{i+1}+\bar{d}^{(i+1,2^{l-1}j-2^{l-2}+1)})}\cdot\\
& &\prod_{i=l-1}^{k-1}\frac{\alpha_{i+1}+\bar{d}^{(i+1,2^{l-1}j-2^{l-2}+1)}}{\bar{d}^{(i,2^{l-1}j-2^{l-2}+1)}}\cdot\frac{q_1}{q_1-\phi(\alpha_1+\ldots+\alpha_{l-1}+\psi(q_l))}\\
&=&c^{(2^{l-1}j-2^{l-2}+1,k-1)}\cdot\prod_{i=2}^k\frac{q_i}{q_i-\phi(\alpha_i+\bar{d}^{(i,2^{l-1}j-2^{l-2}+1)})}\cdot\\
& &\prod_{i=l}^k\frac{\alpha_i+\bar{d}^{(i,2^{l-1}j-2^{l-2}+1)}}{\bar{d}^{(i-1,2^{l-1}j-2^{l-2}+1)}}\frac{q_1}{q_1-\phi(\alpha_1+\ldots+\alpha_{l-1}+\psi(q_l))},
\end{eqnarray*}
where $j=1,2,\ldots,2^{k-l}$; here $\bar{d}^{(i,2^{l-1}j-2^{l-2}+1)}$ is given by 
\[\bar{d}^{(i,2^{l-1}j-2^{l-2}+1)}= \left \{ \begin{array}{ll}
\alpha_{i+1}+\bar{d}^{(i+1,2^{l-1}j-2^{l-2}+1)} &\mbox{ if } \ceil*{\frac{2^{l-1}j-2^{l-2}+1}{2^{i-1}}} \text{ is odd,}\\
\\
\psi(q_{i+1}) &\mbox{ if } \ceil*{\frac{2^{l-1}j-2^{l-2}+1}{2^{i-1}}} \text{ is even}.
\end{array} \right. \]
This table follows from Definition~\ref{DefinitionL} and the observation that the factor $\bar{d}^{(i,2^{l-1}j-2^{l-2}+1)}$ initially was the factor added to the term $\alpha_{i-1}$ (this is why we use the notation $\bar{d}$ for these terms); this is due to the fact that in (\ref{integraal}) all indices are raised by one. In order to bring this into the form of Definition \ref{DefinitionL} we observe the following:
\begin{enumerate}
\item[(a)]
Concerning the signs we have the relation $c^{(2^{l-1}j-2^{l-2}+1,k-1)}=c^{(2^lj-2^{l-1}+1,k)}$ for all $l=2,...,k$ and $j=1,...,2^{k-l}$. We see this as follows. From Lemma \ref{lemmasign} we see it is sufficient to show that the numbers $2^{l-1}j-2^{l-2}$ and $2^lj-2^{l-1}$ have the same parity. But this holds as $2^lj-2^{l-1}=2(2^{l-1}j-2^{l-2})$. Intuitively we can see this from the tree graph in Fig.\ \ref{tree graph 3}; every time we move down and left the sign is always the same.
\item[(b)] Concerning the labeling of the terms, using the fact that
\[\ceil*{\frac{2^{l-1}j-2^{l-2}+1}{2^{i-1}}}=\ceil*{\frac{2^{l}j-2^{l-1}+1}{2^{i}}}\]
(which we obtain from Remark \ref{spirit lance remark}) we obtain
\begin{equation}
\label{genji sword}
d^{(i,2^{l}j-2^{l-1}+1)}=\bar{d}^{(i,2^{l-1}j-2^{l-2}+1)}.
\end{equation}
\item [(c)] From the four properties in Remark \ref{spirit lance remark} we see that $d^{(1,2^lj-2^{l-1}+1)}=\alpha_2+\ldots+\alpha_{l-1}+\psi(q_l)$ and
\[\frac{\alpha_i+d^{(i,2^lj-2^{l-1}+1)}}{d^{(i-1,2^lj-2^{l-1}+1)}}=1,\]
for all $i=1,2,\ldots,l-1$.
\end{enumerate}
The arguments in (a)-(c) show that, for $l=2,3,\ldots,n$, $j=1,2,\ldots,2^{k-l}$,
\begin{align*}
L_{(2^lj-2^{l-1}+1,l)}^{(k)} =& c^{(j, k)}\cdot\prod_{i=1}^k\frac{q_i}{q_i-\phi(\alpha_i+d^{(i,2^lj-2^{l-1}+1)})}\cdot\prod_{i=l}^{k}\frac{\alpha_i+d^{(i,2^lj-2^{l-1}+1)}}{d^{(i-1,2^lj-2^{l-1}+1)}},
\end{align*}
where the $d^{(i,2^lj-2^{l-1}+1)}$ are given by the table in Definition \ref{DefinitionL}.

\bigskip \noindent
$\circ$ {\it Coefficient of $\mathcal{I}_2$.} \hspace{1mm}
For the terms $L_{(2j,1)}^{(k)}$, $j=1,2,\ldots,2^{k-1}$ (i.e., the coefficients of $\exp[{-\psi(q_1)x}]$ for $k$ exponentially distributed random variables) we observe that these are given from all terms in the previous step, one from each (this corresponds to moving down and right in the tree graph in Fig.\ \ref{tree diagram 1} or Fig.\ \ref{tree graph 2}). The first term, $L_{(2,1)}^{(k)}$ results from the integration
\[\int_0^\infty \prod_{i=1}^{k-1}\frac{q_{i+1}}{q_{i+1}-\phi(\alpha_{i+1}+\ldots+\alpha_{k})}e^{-(\alpha_1+\ldots+\alpha_k)y}\Pb_x(\Q_{T_1}\in{\rm d}y),\]
which leads to
\[L_{(2,1)}^{(k)} = -\prod_{i=2}^k\frac{q_i}{q_i-\phi(\alpha_i+\ldots+\alpha_k)}\frac{q_1}{q_1+\phi(\alpha_1+\ldots+\alpha_k)}\frac{\alpha_1+\ldots+\alpha_k}{\psi(q_1)}.\]
Since $l=1$ and $j=1$, we have for $i=1,2,\ldots,k$
\[\ceil*{\frac{2}{2^i}}=1,    \]
showing that $d^{(i,2)} = \sum_{s=i+1}^k\alpha_s$. Furthermore, we see that for all $i=2,3,\ldots,k$
\[\frac{\alpha_i+d^{(i,1)}}{d^{(i-1,2)}}=1, \]
and, hence, we get
\[\prod_{i=1}^k\frac{\alpha_i+d^{(i,2)}}{d^{(i-1,2)}}= \frac{\alpha_1+d^{(1,2)}}{d^{(0,2)}}=\frac{\alpha_1+\ldots+\alpha_k}{\psi(q_1)}.\]
By using these facts, it follows that
\begin{equation}
\label{Poseidon equation}
L_{(2,1)}^{(k)} = -\prod_{i=1}^k\frac{q_i}{q_i-\phi(\alpha_i+d^{(i,2
)})}\prod_{i=1}^k\frac{\alpha_i+d^{(i,2)}}{d^{(i-1,2)}},
\end{equation}
corresponding to Definition \ref{DefinitionL} and Lemma \ref{lemmasign}.

\bigskip \noindent
$\circ$ {\it Coefficient of $\mathcal{II}_2$.} \hspace{1mm}
In general, the terms $L_{(2^{l+1}j-2^l+2,1)}^{(k)}$, for $l=1,2,\ldots,k-1$ and $j=1,2,\ldots,2^{k-1-l}$, are derived from the integrals
\begin{equation}
\label{integral}
\int_0^\infty L_{(2^lj-2^{l-1}+1,l)}^{(k-1)}e^{-\alpha_1y}e^{-(\alpha_2+\ldots+\alpha_{l}+\psi(q_{l+1}))y}\Pb_x(\Q_{T_1}\in{\rm d}y).
\end{equation}
Consider the terms $L_{(2^{l+1}j-2^l+2,1)}^{(k)}$ for $l=1,\ldots,k-1$ and $j=1,\ldots,2^{k-1-l}$.
%; observe that these can be  derived from the terms $L_{(2^lj-2^{l-1}+1,l)}^{(k-1)}$.  
From the integral in (\ref{integral}) we obtain, for $l=1,\ldots,k-2$ and $j=1,2,\ldots,2^{k-2-l}$, that
\begin{align*}
L_{(2^{l+1}j-2^l+2,1)}^{(k)}=&-c^{(2^lj-2^{l-1}+1,k-1)}\cdot\prod_{i=1}^{k-1}\frac{q_{i+1}}{q_{i+1}-\phi(\alpha_{i+1}+\bar{d}^{(i+1,2^{l}j-2^{l-1}+1)})} \\
&\cdot\prod_{i=l}^{k-1}\frac{\alpha_{i+1}+\bar{d}^{(i+1,2^lj-2^{l-1}+1)}}{\bar{d}^{(i,2^lj-2^{l-1}+1)}}\cdot\frac{q_1}{q_1-\phi(\alpha_1+...+\alpha_l+\psi(q_{l+1}))}\cdot\frac{\alpha_1+...+\alpha_l+\psi(q_{l+1})}{\psi(q_1)}\\
=&-c^{(2^lj-2^{l-1}+1,k-1)}\cdot\prod_{i=2}^{k}\frac{q_{i}}{q_{i}-\phi(\alpha_{i}+\bar{d}^{(i,2^{l}j-2^{l-1}+1)})} \\
&\cdot\prod_{i=l+1}^{k}\frac{\alpha_{i}+\bar{d}^{(i,2^lj-2^{l-1}+1)}}{\bar{d}^{(i-1,2^lj-2^{l-1}+1)}}\cdot\frac{q_1}{q_1-\phi(\alpha_1+\ldots+\alpha_l+\psi(q_{l+1}))}\cdot\frac{\alpha_1+\ldots+\alpha_l+\psi(q_{l+1})}{\psi(q_1)},
\end{align*}
where the factors $\bar{d}^{(i,2^{l}j-2^{l-1}+1)}$ are given by 
\[ \bar{d}^{(i,2^{l}j-2^{l-1}+1)}= \left\{ \begin{array}{ll}
\alpha_{i+1}+\bar{d}^{(i+1,2^{l}j-2^{l-1}+1)} &\mbox{ if } \ceil*{\frac{2^{l}j-2^{l-1}+1}{2^{i-1}}} \text{ is odd,}\\
\\
\psi(q_{i+1}) &\mbox{ if } \ceil*{\frac{2^{l}j-2^{l-1}+1}{2^{i-1}}} \text{ is even}.
\end{array} \right. \]
Using the same observation as in (\ref{genji sword}), it is found, for $j=1,\ldots,2^{k-l-1}$ and $i=l+1,...,k$, that
\[d^{(i,2^{l+1}j-2^{l}+2)}=\bar{d}^{(i,2^lj-2^{l-1}+1)}.\]
From Remark \ref{spirit lance remark} (a)-(c) we see that
\[d^{(1,2^{l+1}j-2^l+2)}=\alpha_2+\ldots+\alpha_l+\psi(q_{l+1}),\:\:\:\:\:\:d^{(0,2^{l+1}j-2^l+2)}=\psi(q_1)\]
and, for $i=2,3,\ldots,l$,
\[\frac{\alpha_i+d^{(i,2^{l+1}j-2^l+2)}}{d^{(i-1,2^{l+1}j-2^l+2)}}=1.\]
These observations allow us to write  $L_{(2^{l+1}j-2^l+2,1)}^{(k)}$ as follows
\[L_{(2^{l+1}j-2^l+2,1)}^{(k)}=-c^{(2^lj-2^{l-1}+1,k-1)}\cdot\prod_{i=1}^{k}\frac{q_{i}}{q_{i}-\phi(\alpha_{i}+d^{(i,2^{l+1}j-2^{l}+2)})}\cdot\prod_{i=1}^{k}\frac{\alpha_{i}+d^{(i,2^{l+1}j-2^{l}+2)}}{d^{(i-1,2^{l+1}j-2^{l}+2)}}.\]
Concerning the signs, we obtain the relation $c^{(2^{l+1}j-2^l+2,k)}=-c^{(2^lj-2^{l-1}+1,k-1)}$ since the numbers $2^{l+1}j-2^l+1=2(2^lj-2^{l-1})+1$ and $2^lj-2^{l-1}$ have opposite parities. This final expression agrees with those presented in Thm.\ \ref{MainTheorem}.

\bigskip \noindent
Now, we combine the above results to complete the proof. Using the coefficients of $\mathcal{I}_1$ and $\mathcal{I}_2$, i.e., (\ref{Dias equation}) and (\ref{Poseidon equation}), we can rewrite (\ref{Hades equation}) to
\begin{align*}
\mathcal{L}=&\prod_{i=1}^k\frac{q_i}{q_i-\phi(\alpha_i+...+\alpha_k)}e^{-(\alpha_1+\ldots+\alpha_k)x}-\prod_{i=1}^k\frac{q_i}{q_i-\phi(\alpha_i+d^{(i,2)})}\prod_{i=1}^k\frac{\alpha_i+d^{(i,2)}}{d^{(i-1,2)}}e^{-\psi(q_1)x}\\
&+\sum_{l=2}^{k}\sum_{j=1}^{2^{k-l}}L_{(2^{l-1}j-2^{l-2}+1,l-1)}^{(k-1)}\int_0^\infty e^{-(\alpha_1+\ldots+\alpha_{l-1}+\psi(q_l))y}\Pb_x(\Q_{T_1}\in{\rm d}y).
\end{align*}
Using the definition of $L_{(2,1)}^{(k)}$, in conjunction with the coefficients $\mathcal{II}_1$ and $\mathcal{II}_2$ and Definition \ref{DefinitionL}, the above expression can be written as 
\begin{align*}
\mathcal{L}=&\prod_{i=1}^k\frac{q_i}{q_i-\phi(\alpha_i+...+\alpha_k)}e^{-(\alpha_1+\ldots+\alpha_k)x}-L_{(2,1)}^{(k)}e^{-\psi(q_1)x}\\
&+\sum_{l=2}^{k}\sum_{j=1}^{2^{k-l}}L_{(2^lj-2^{l-1}+1,l)}^{(k)}e^{-(\alpha_1+\ldots+\alpha_{l-1}+\psi(q_l))x}\hspace{-2pt}+\hspace{-2pt}\sum_{l=2}^k\sum_{j=1}^{2^{k-l}}L_{(2^lj-2^{l-1}+2,1)}^{(k)}e^{-\psi(q_1)x}.
\end{align*}
It remains to write the last sum in the desired form. This double sum has in total $2^{k-1}-1$ terms, and we observe that for $l=2,\ldots,k$ and $j=1,\ldots,2^{k-l}$,  $2^lj-2^{l-1}+2$ defines a partition of the even numbers $4,6,\ldots,2^k$ into $k-1$ classes each one containing $2^{k-l}$ numbers. Relabeling the terms with only one subscript, we can write this double sum as
\[\sum_{l=2}^k\sum_{j=1}^{2^{k-l}}L_{(2^lj-2^{l-1}+2,1)}^{(k)}e^{-\psi(q_1)x}=\sum_{i=2}^{2^{k-1}}L_{(2i,1)}^{(k)}e^{-\psi(q_1)x},\]
where $i=2^lj-2^{l-1}+2$ for $l=2,\ldots,k$ and $j=1,\ldots,2^{k-l}$. 
From the above it follows that $\mathcal{L}$ can be written as the expression in Thm.~\ref{MainTheorem}. This completes the proof.

\end{proof}

\subsection{Proof of Theorem \ref{spectrally negative general case density}}

For the spectrally negative case we use a tree graph to illustrate, similar to Fig.\ \ref{tree graph 2}, how the coefficients of the convolution terms change from step $n$ to $n+1$. Based on the reasoning that led us to the tree diagram in Fig.\ \ref{tree spectrally negative} for the spectrally negative case, we adopt the numbering of the coefficients $L_{(2^lj-2^{l-1}+1,l)}^{(n)}$ as in Fig.~\ref{tree graph coefficients spectrally negative}.

% Set the overall layout of the tree
\tikzstyle{level 1}=[level distance=1.5cm, sibling distance=8cm]
\tikzstyle{level 2}=[level distance=1.5cm, sibling distance=4cm]
\tikzstyle{level 3}=[level distance=1.5cm, sibling distance=2cm]

% Define styles for bags and leafs
\tikzstyle{bag} = [text width=5.5em, text centered]
\tikzstyle{end} = [circle, minimum width=7pt,fill, inner sep=0pt]

\begin{figure}[H]
\begin{tikzpicture}[grow=down, sloped]
\node[bag] {}
    child {
        node[bag] {$L_1^{(1)}$}
            child {
                node[bag] {$L_1^{(2)}$}
                     child{
                        node[bag] {$L_1^{(3)}$}
                        edge from parent
                        node[above] {$$}
                        node[below]  {$$}
                     }
                     child{
                       node[bag] {$L_{5,3}^{(3)}$}
                        edge from parent
                        node[above] {$$}
                        node[below]  {$$}
                     }
                    edge from parent
                    node[above] {$$}
                    node[below]  {$$}
            }
            child {
                node[bag] {$L_{3,2}^{(2)}$}
                    child{
                        node[bag] {$L_{3,2}^{(3)}$}
                        edge from parent
                        node[above] {$$}
                        node[below]  {$$}
                     }
                     child{
                       node[bag] {$L_{7,2}^{(3)}$}
                         edge from parent
                         node[above] {$$}
                         node[below]  {$$}
                     }
                edge from parent
                node[above] {$$}
                node[below]  {$$}
            }
            edge from parent
            node[above] {}
            node[below]  {$$}
    }
    child {
        node[bag] {$L_{2,1}^{(1)}$}
            child {
                node[bag] {$L_{2,1}^{(2)}$}
                     child{
                        node[bag] {$L_{2,1}^{(3)}$}
                        edge from parent
                        node[above] {$$}
                        node[below]  {$$}
                     }
                     child{
                       node[bag] {$L_{6,1}^{(3)}$}
                        edge from parent
                        node[above] {$$}
                        node[below]  {$$}
                     }
                edge from parent
                node[above] {$$}
                node[below]  {$$}
            }
            child {
                node[bag] {$L_{4,1}^{(2)}$}
                  child{
                        node[bag] {$L_{4,1}^{(3)}$}
                        edge from parent
                node[above] {$$}
                node[below]  {$$}
                     }
                  child{
                       node[bag] {$L_{8,1}^{(3)}$}
                     edge from parent
                node[above] {$$}
                node[below]  {$$}
                     }
                edge from parent
                node[above] {$$}
                node[below]  {$$}
            }
            edge from parent
            node[above] {}
            node[below]  {$$}
    };

\end{tikzpicture}
\caption{The coefficients of the convolution terms}
\label{tree graph coefficients spectrally negative}
\end{figure}
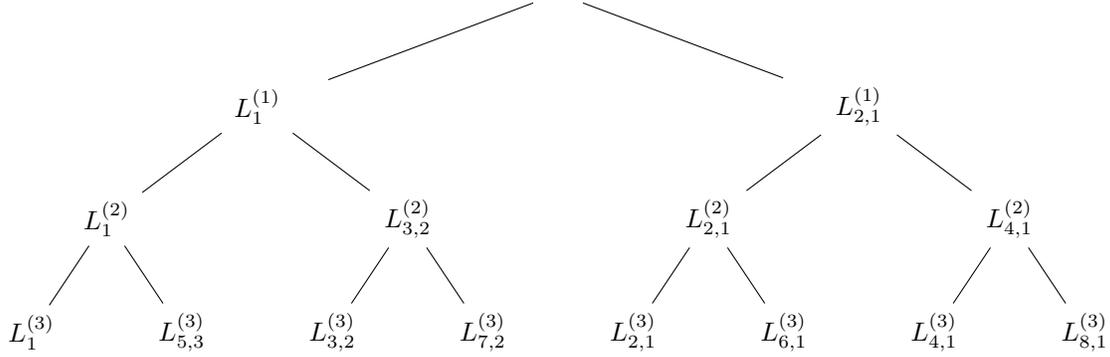
Comparing this tree graph with the one corresponding to the spectrally positive case (i.e.,  Fig.\ \ref{tree graph 2}), we observe that now the numbering of the coefficients is `mixed'. In the spectrally positive case every time we went right a new subtree with root $\psi(q_1)$ was generated, whereas here only the first time we turn right a subtree with root $\left(Z^{(q_l)}\star W^{(q_{l-1})}\star\ldots\star W^{(q_1)}\right)(x)$ is generated. So, in the spectrally negative case, convolution terms of the same order gather in one subtree.  When moving downwards in the subtree the coefficients change according to a recursive pattern. This mixed enumeration of the coefficients is due to these two characteristics. In the spectrally positive case terms of all possible orders are generated in all subtrees generated by some $\psi(q_1)$. From Fig.\ \ref{tree graph coefficients spectrally negative} we see that the coefficient $L_{(2^lj-2^{l-1}+1,l)}^{(n)}$ generates the coefficients $L_{(2^lj-2^{l-1}+1,l)}^{(n+1)}$ and $L_{(2^lj-2^{l-1}+1+2^{n},l)}^{(n+1)}$. At step $n$ we observe that, for $l=1,\ldots,n$, we have a total of $2^{n-l}$ terms of type $\left(Z^{(q_l)}\star W^{(q_{l-1})}\star\ldots\star W^{(q_1)}\right)(x)$, which is why we choose to {\em label} the coefficients of these terms by the numbers $2^lj-2^{l-1}+1$. This {\em trick} leads to an expression which is relatively easy to work with in the proof, and at the same time has a structure similar to the one featuring in Thm.\ \ref{MainTheorem} for the spectrally positive case. 

Let us now first consider the sign of the $j$-th term when we have $n$ exponentially distributed random variables. From Thm.\ \ref{spectrally negative one exponential time} we see that, for $n=1$, the signs of the coefficients are $-,+$. For $n=2$ and the expression in Eqn.\ (\ref{expression for density spectrally negative two times final}) it turns out  that the signs are $+,-,-,+$. Following how the terms are produced when we go from the step with $n$ exponential times to the step with $n+1$ exponential times (Section~\ref{section spectrally negative} and Fig.\ \ref{tree graph coefficients spectrally negative}), the signs at every step can be represented by the tree graph in Fig.~\ref{tree graph 1}.  In every row, starting from left to right, the nodes represent the sign of every factor when our expression is written in the form presented in Fig.\ \ref{tree graph coefficients spectrally negative}.

% Set the overall layout of the tree
\tikzstyle{level 1}=[level distance=1cm, sibling distance=7cm]
\tikzstyle{level 2}=[level distance=1cm, sibling distance=3cm]
\tikzstyle{level 3}=[level distance=1cm, sibling distance=1.5cm]
\tikzstyle{level 4}=[level distance=1cm, sibling distance=0.5cm]

% Define styles for bags and leafs
\tikzstyle{bag} = [text width=8em, text centered]
\tikzstyle{end} = [circle, minimum width=7pt,fill, inner sep=0pt]

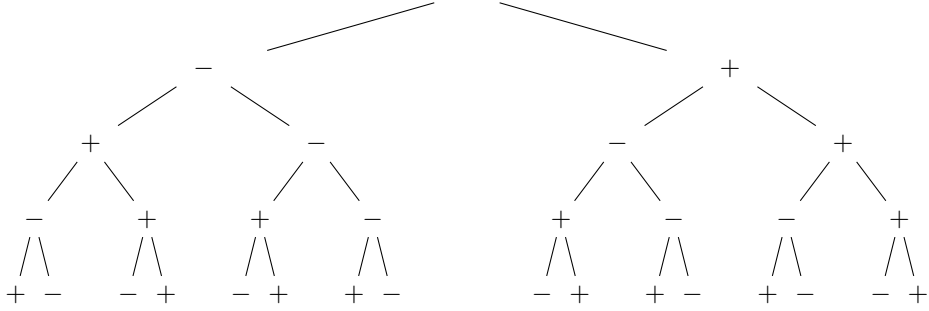
\begin{figure}[H]
\begin{tikzpicture}[grow=down, sloped]
\node[bag] {}
    child {
        node[bag] {$-$}
            child {
                node[bag] {$+$}
                     child{
                        node[bag] {$-$}
                           child{
                                node [bag] {$+$}
                                edge from parent
                                node[above] {$$}
                                node[below]  {$$}
                           }
                            child{
                                node [bag] {$-$}
                                edge from parent
                                node[above] {$$}
                                node[below]  {$$}
                           }
                        edge from parent
                        node[above] {$$}
                        node[below]  {$$}
                     }
                     child{
                       node[bag] {$+$}
                          child{
                                node [bag] {$-$}
                                edge from parent
                                node[above] {$$}
                                node[below]  {$$}
                           }
                            child{
                                node [bag] {$+$}
                                edge from parent
                                node[above] {$$}
                                node[below]  {$$}
                           }
                        edge from parent
                        node[above] {$$}
                        node[below]  {$$}
                     }
                    edge from parent
                    node[above] {$$}
                    node[below]  {$$}
            }
            child {
                node[bag] {$-$}
                    child{
                        node[bag] {$+$}
                            child{
                                node [bag] {$-$}
                                edge from parent
                                node[above] {$$}
                                node[below]  {$$}
                           }
                            child{
                                node [bag] {$+$}
                                edge from parent
                                node[above] {$$}
                                node[below]  {$$}
                           }
                        edge from parent
                        node[above] {$$}
                        node[below]  {$$}
                     }
                     child{
                       node[bag] {$-$}
                              child{
                                node [bag] {$+$}
                                edge from parent
                                node[above] {$$}
                                node[below]  {$$}
                           }
                            child{
                                node [bag] {$-$}
                                edge from parent
                                node[above] {$$}
                                node[below]  {$$}
                           }
                         edge from parent
                         node[above] {$$}
                         node[below]  {$$}
                     }
                edge from parent
                node[above] {$$}
                node[below]  {$$}
            }
            edge from parent
            node[above] {}
            node[below]  {$$}
    }
    child {
        node[bag] {$+$}
            child {
                node[bag] {$-$}
                     child{
                        node[bag] {$+$}
                            child{
                                node [bag] {$-$}
                                edge from parent
                                node[above] {$$}
                                node[below]  {$$}
                           }
                            child{
                                node [bag] {$+$}
                                edge from parent
                                node[above] {$$}
                                node[below]  {$$}
                           }
                        edge from parent
                        node[above] {$$}
                        node[below]  {$$}
                     }
                     child{
                       node[bag] {$-$}
                            child{
                                node [bag] {$+$}
                                edge from parent
                                node[above] {$$}
                                node[below]  {$$}
                           }
                            child{
                                node [bag] {$-$}
                                edge from parent
                                node[above] {$$}
                                node[below]  {$$}
                           }
                        edge from parent
                        node[above] {$$}
                        node[below]  {$$}
                     }
                edge from parent
                node[above] {$$}
                node[below]  {$$}
            }
            child {
                node[bag] {$+$}
                  child{
                        node[bag] {$-$}
                            child{
                                node [bag] {$+$}
                                edge from parent
                                node[above] {$$}
                                node[below]  {$$}
                           }
                            child{
                                node [bag] {$-$}
                                edge from parent
                                node[above] {$$}
                                node[below]  {$$}
                           }
                        edge from parent
                node[above] {$$}
                node[below]  {$$}
                     }
                  child{
                       node[bag] {$+$}
                          child{
                                node [bag] {$-$}
                                edge from parent
                                node[above] {$$}
                                node[below]  {$$}
                           }
                            child{
                                node [bag] {$+$}
                                edge from parent
                                node[above] {$$}
                                node[below]  {$$}
                           }
                     edge from parent
                node[above] {$$}
                node[below]  {$$}
                     }
                edge from parent
                node[above] {$$}
                node[below]  {$$}
            }
            edge from parent
            node[above] {}
            node[below]  {$$}
    };
\end{tikzpicture}
\caption{The sequence of the signs at every step}
\label{tree graph 1}

\end{figure}

We see that row $n+1$ can be obtained from row $n$ if we substitute every $+$ by the pair $-,+$, and every $-$ by the pair $+,-$. 
This holds due to the mechanism analyzed in Section \ref{section spectrally negative}, i.e., the order in which the integrations are carried out. 
Denote by $c^{(j,n)}$ the sign of the $j$-th element in the $n$-th row in the tree. Then $j=1,2,\ldots,2^n$ and $c^{(j,n)}$ corresponds to the sign of the $j$-th coefficient when we have $n$ exponentially distributed random variables in the expression considered in Thm.\ \ref{spectrally negative general case density}.

\begin{proof}[Proof of Lemma \ref{lemma sign 2}]
We prove this lemma by induction.
\begin{enumerate}
\item[(i)] For $n=1$ we have to find the values of $c^{(1,1)}$ and $c^{(2,1)}$. For $j=1$ we need the binary expansion of $2^1 - 1=1$, which has one 1, while for $j=2$ we need the binary expansion of $2^1-2=0$ which has zero ones. Thus we get that $c^{(1,1)}=-1$ and $c^{(2,1)}=+1$.
\item[(ii)] We assume the lemma holds for $n=k$, i.e., for row $k$ of the tree graph in Fig.\ \ref{tree graph 1}. Hence, for $j=1,2,\ldots,2^k$,
\[c^{(j,k)} = (-1)^{\PAR\{\beta_0,\ldots,\beta_{k-1}\}}.\]
From the tree presented above we observe that the $2^n$ signs of an arbitrary row are the same as the last $2^n$ signs of row $n+1$.

Consider now the $(k+1)$-th row of the tree. 
Using the induction hypothesis and the observation above it follows that the lemma holds for the last $2^k$ signs of the $(k+1)$-th row as well. 
We can also see this by observing that for the last $2^k$ signs of the $(k+1)$-th row we are interested in the binary expansions of $2^{k+1} - j$ for $j=2^k+1,\ldots,2^{k+1}$, which is essentially equivalent to considering the binary expansions of $2^k-j$ for $j=1,\ldots,2^k$. This shows that, for $j=1,\ldots,2^k$, 
\[c^{(j,k)}=c^{(j+2^k,k+1)}.\]

 What remains is to prove the lemma for the elements $2^{k}-j$, $j=1,\ldots,2^k$, of the $(k+1)$-th row. At this point, we observe that at an arbitrary row $n$, because of symmetry\[c^{(j,n)} = - c^{(j+2^{n-1},n)}.\]
Hence, the signs of terms $j$ and $j+2^{n-1}$ in the $n$-th row will always be opposite. This yields that in the $(k+1)$-th row we  have, for $j=1,2,\ldots,2^k$,
\[c^{(j,k+1)} = - c^{(j+2^k,k+1)}.\]
But we know that
\[c^{(j+2^k,k+1)} = (-1)^{\PAR\{\beta_0,\ldots,\beta_{k-1},0\}}.\]
We also know that, for $j=1,2,...,2^k$, the binary representation of $2^{k+1}-j$ has one more 1 than the binary representation of $2^{k+1} - 2^k-j = 2^k-j$. This leads to the expression
\[(-1)^{\PAR\{\beta_0,\ldots,\beta_{k-1},1\}} = - (-1)^{\PAR\{\beta_0,\ldots,\beta_{k-1},0\}}\]
and in addition
\[c^{(j,k+1)} = (-1)^{\PAR\{\beta_0,...,\beta_{k}\}},\]
for all $j=1,2,\ldots,2^{k+1}$.
\end{enumerate}
\end{proof}

\noindent
Before proceeding to the proof of Thm.\ \ref{spectrally negative general case density} we make some remarks concerning the result established in Lemma \ref{lemma sign 2}; these remarks are used in the proof of Thm.\ \ref{spectrally negative general case density}.
\begin{remark}{\em 
\label{first remark}
For an arbitrary row $n$ in the tree presented in Fig.\ \ref{tree graph 1}, we have that
\[c^{(1,n)} = - c^{(1,n+1)}.\]
We know that in order to find the first sign of the $n$-th row we must find the binary expansion of the element $2^n - 1$, which has exactly $n$ ones. Thus, for an arbitrary $n\geq1$, we get the expression
\[c^{(1,n)} = (-1) ^n\]
and this also shows the relation mentioned in the remark.}
\end{remark}
\begin{remark}{\em 
\label{second remark}
For $l=1,...,n$ and $j=1,\ldots,2^{n-l}$ we have
\[c^{(2^lj-2^{l-1}+1,n)}=-c^{(2^lj-2^{l-1}+1,n+1)}.\]
To see this we observe, as in the proof of Lemma \ref{lemma sign 2}, that the $2^{n}$ signs of the $n$-th row are the same as the last $2^{n}$ signs of the $(n+1)$-th row. This gives, for $l=1,...,n$ and $j=1,\ldots,2^{n-l}$,
\begin{equation}
\label{signs 3rd remark}
c^{(2^lj-2^{l-1}+1,n)} = c^{(2^lj-2^{l-1}+2^{n-1}+1,n+1)}.
\end{equation}
Using the symmetry of the signs in each row, i.e., for $i=1,\ldots,2^{n-1}$,
$c^{(i,n)} = -c^{(i+2^{n-1},n)}$,
we obtain the equality above.
%\[c^{(2^lj-2^{l-1}+1,n)}=-c^{(2^lj-2^{l-1}+1,n+1)}.\]
}
\end{remark}
\begin{remark}{\em 
\label{third remark}
For the $j$-th sign of the $n$-th row and the $(2^{n}+j)$-th sign of the $(n+1)$-th row we have the following expression,
\begin{equation}
\label{sign first}
c^{(j,n)} = c^{(2^{n}+j,n+1)}.
\end{equation}
For the value of $c^{(j,n)}$ we need the binary representation of $2^{n}-j$ while for the value of $c^{(2^{n}+j,n+1)}$ we need the binary representation of $2^{n+1} - 2^{n}-j = 2^{n}-j$, giving (\ref{sign first}).}
\end{remark}

\begin{proof}[Proof of Theorem \ref{spectrally negative general case density}]
Now, having Lemma \ref{lemma sign 2} at our disposal, we proceed with the proof of Thm.\ \ref{spectrally negative general case density}. Consider the case $n=1$. Then we have one exponentially distributed random variable $T_1$ with parameter $q_1$. Thm.\ \ref{spectrally negative general case density} gives
\[\Pb_x(\Q_{T_1}\in{\rm d}y)=\Big[ c^{(1,1)}q_1W^{(q_1)}(x-y)+L_{(2,1)}^{(1)}Z^{(q_1)}(x)\Big]{\rm d}y.\]
From Lemma \ref{lemma sign 2} we have that $c^{(1,1)}=-1$ and $c^{(2,1)}=+1$. Due to Definition \ref{lambdaexpo}, $L_{(2,1)}^{(1)}=+\Psi(q_1)e^{-\Psi(q_1)y}$. Hence we obtain
\[\Pb_x(\Q_{T_1}\in{\rm d}y)=\Big[-q_1W^{(q_1)}(x-y)+\Psi(q_1)e^{-\Psi(q_1)y}Z^{(q_1)}(x)\Big]{\rm d}y.\]
This is the expression found in (\ref{spectrally negative one exponential time}), and we conclude that our result holds for $n=1$.

\bigskip \noindent
We now assume that Thm.\ \ref{spectrally negative general case density} holds for $n=k-1$. Consider now the case of $n=k$ exponentially distributed random variables, then conditioning on the value of the workload at the first $k-1$ exponential epochs yields
\begin{eqnarray}
 & & \hspace{-5em} \Pb_x(\Q_{T_1+...+T_k}\in{\rm d}y) \nonumber \\
& = & \int_{z=0}^\infty \Pb_x(\Q_{T_1+...+T_{k-1}}\in{\rm d}z)\Pb_z(\Q_{T_{k}}\in{\rm d}y) \nonumber \\
& = & \int_{z=0}^\infty \Biggl( c^{(1,k-1)}\prod_{i=1}^{k-1}q_i\left(W^{(q_{k-1})}\star...\star W^{(q_1)}\right)(x-z) \nonumber \\
& & +\sum_{l=1}^{k-1}\sum_{j=1}^{2^{k-l-1}}L_{(2^lj-2^{l-1}+1,l)}^{(k-1)}(z) 
\left(Z^{(q_l)}\star W^{(q_{l-1})}\star...\star W^{(q_1)}\right)(x) \Biggr) \Pb_z(\Q_{T_{k}}\in{\rm d}y){\rm d}z \nonumber \\
&=& \Bigg[c^{(1,k-1)}\prod_{i=1}^{k-1}q_i\int_{z=0}^\infty\left(W^{(q_{k-1})}\star...\star W^{(q_1)}\right)(x-z)\left(-q_{k}W^{(q_{k})}(z-y)\right){\rm d}z \nonumber\\
& & +c^{(1,k-1)}\prod_{i=1}^{k-1}q_i\Psi(q_{k})e^{-\Psi(q_{k})y}\int_{z=0}^\infty \left(W^{(q_{k-1})}\star...\star W^{(q_1)}\right)(x-z)Z^{(q_{k})}(z){\rm d}z \nonumber \\
& & +\sum_{l=1}^{k-1}\sum_{j=1}^{2^{k-l-1}}\Bigg(\int_0^\infty L_{(2^lj-2^{l-1}+1,l)}^{(k-1)}(z)\left(-q_{k}W^{(q_{k})}(z-y)\right){\rm d}z \nonumber \\
& & +\Psi(q_{k})e^{-\Psi(q_{k})y} \int_0^\infty L_{(2^lj-2^{l-1}+1,l)}^{(k-1)}(z)Z^{(q_{k})}(z){\rm d}z\Bigg)\left(Z^{(q_l)}\star...\star W^{(q_1)}\right)(x)\Bigg]{\rm d}y.  \label{huge expression with 4 integrals}
\end{eqnarray}
We see that we have to evaluate the following four integrals:
\begin{eqnarray}
{\mathcal J}_1 &=& \int_{z=0}^\infty \left(W^{(q_{k-1})}\star...\star W^{(q_1)}\right)(x-z)W^{(q_{k})}(z-y){\rm d}z, \label{integral 1 spectrally negative} \\
{\mathcal J}_2 &=& \int_{z=0}^\infty \left(W^{(q_{k-1})}\star...\star W^{(q_1)}\right)(x-z)Z^{(q_{k})}(z){\rm d}z, \label{integral 2 spectrally negative} \\
{\mathcal J}_3 &=& -q_{k}\int_{z=0}^\infty L_{(2^lj-2^{l-1}+1,l)}^{(k-1)}(z)W^{(q_{k})}(z-y){\rm d}z, \label{integral 3 spectrally negative} \\
{\mathcal J}_4 &=& \Psi(q_{k})e^{-\Psi(q_{k})y}\int_{z=0}^\infty L_{(2^lj-2^{l-1}+1,l)}^{(k-1)}(z)Z^{(q_{k})}(z){\rm d}z. \label{integral 4 spectrally negative}
\end{eqnarray}
These integrals ${\mathcal J}_1,\ldots,{\mathcal J}_4$ can be interpreted by looking at the tree graph in Fig.\ \ref{tree graph coefficients spectrally negative}. ${\mathcal J}_1$ is related to the scenario we are at node 1 of row $k-1$ and we move down and left, while ${\mathcal J}_2$ is related to the scenario we move down and right. ${\mathcal J}_3$ suggests that we are at some node in row $k-1$ which lies in a subtree with root $k-l-1$ rows above the node under consideration and we move left, while ${\mathcal J}_4$ corresponds to the case we move down and right. 

\vb

\noindent
$\circ$ {\it Integral ${\mathcal J}_1$.} \hspace{1mm} By a change of variable argument and using the fact that $W^{(q)}(x)=0$ for $x<0$, we find that
\begin{equation}
\label{J1 integral}
{\mathcal J}_1 = \left(W^{(q_{k})}\star\ldots\star W^{(q_1)}\right)(x-y).
\end{equation}
From (\ref{huge expression with 4 integrals}) we see that the first term of $\Pb_x(\Q_{T_1+\ldots+T_{k}}\in{\rm d}y)$ is equal to
\[-c^{(1,k-1)}\prod_{i=1}^{k}q_i\left(W^{(q_{k})}\star\ldots\star W^{(q_1)}\right)(x-y).\]
Using Remark \ref{first remark} we have $-c^{(1,k-1)}=c^{(1,k)}$ and this shows that we have identified the first term of the expression in Thm.~\ref{spectrally negative general case density}. 

\vb

\noindent
$\circ$ {\it Integral ${\mathcal J}_2$.} \hspace{1mm} 
It is straightforward that
\begin{equation}
\label{J2 integral}
{\mathcal J}_2 = \left(Z^{(q_{k})}\star W^{(q_{k-1})}\star...\star W^{(q_1)}\right)(x).
\end{equation}
Hence, the second term of (\ref{huge expression with 4 integrals}) is equal to
\[c^{(1,k-1)}\prod_{i=1}^{k-1}q_i\Psi(q_{k})\exp[{-\Psi(q_{k})y}]\left(Z^{(q_{k})}\star W^{(q_{k-1})}\star\ldots\star W^{(q_1)}\right)(x).\]
This term corresponds to $l=k$ in the summation of the second term in Thm.~\ref{spectrally negative general case density}.
We need to show that \[L_{(2^{k-1}+1,k)}^{(k)}=c^{(1,k-1)}\prod_{i=1}^{k-1}q_i\Psi(q_{k})\exp[{-\Psi(q_{k})y}].\] First we observe that, for $l=k$ and $j=1$, we have that $m(1,k)=k$ and due to Remark \ref{third remark} we have that $c^{(1,k-1)}=c^{(2^{k-1}+1,k)}$. This shows that
\[c^{(1,k-1)}\prod_{i=1}^{k-1}q_i\Psi(q_{k})e^{-\Psi(q_{k})y}=c^{(2^{k-1}+1,k)}\prod_{\mathclap{\substack{i=1,\\ i\neq m(1,k+1)}}}^{k}q_i\Psi(q_{k})e^{-\Psi(q_{k})y}=L_{(2^{k-1}+1,k)}^{(k)},\]
yielding that the expression for the second term, for $l=k$, agrees with Thm.\ \ref{spectrally negative general case density}.

\vb

The integrals ${\mathcal J}_3$ and ${\mathcal J}_4$ correspond to convolution terms in one of the subtrees mentioned before (see Fig.~\ref{tree graph coefficients spectrally negative}). Therefore, the order of the convolution term does not change and these integrations only change the coefficients $L$. 

\vb

\noindent
$\circ$ {\it Integral ${\mathcal J}_3$.} \hspace{1mm} 
We now show that ${\mathcal J}_3$ equals $L_{(2^lj-2^{l-1}+1,l)}^{(k)}(y)$.
%\[J_3=\int_{z=0}^\infty -q_{k}L_{(2^lj-2^{l-1}+1,l)}^{(k-1)}(z)W^{(q_{k})}(z-y){\rm d}z=L_{(2^lj-2^{l-1}+1,l)}^{(k)}(y).\]
In Fig.\ \ref{tree graph coefficients spectrally negative} we see that when we move to the left the numbering of the coefficients remains the same. When looking in row $k-1$ at the element $L_{(2^lj-2^{l-1}+1,l)}^{(k-1)}$, for $l=1,\ldots,k-1$ and $j=1,\ldots,2^{k-l-1}$, and moving down and left we arrive at the element $L_{(2^lj-2^{l-1}+1,l)}^{(k)}$ which has the same `labeling'. 

Since we have an expression for $L_{(2^lj-2^{l-1}+1,l)}^{(k-1)}(z)$, the induction hypothesis yields
\begin{eqnarray}
{\mathcal J}_3 &=&-c^{(2^lj-2^{l-1}+1,k-1)}\Psi(q_{m(j,l)})\prod_{\mathclap{\substack{i=1,\\ i\neq m(j,l)}}}^{k}q_i\prod_{i=l}^{m(j,l)}\frac{1}{q_{i}-q_{i+1}}\prod_{i=m(j,l)+1}^{k-2}\frac{1}{q_{m(j,l)}-q_{i+1}} \nonumber \\
& & \cdot\int_{z=0}^\infty e^{-\Psi(q_{m(j,l)})z}W^{(q_{k})}(z-y){\rm d}z.  \label{integral J_3}
\end{eqnarray}
By a change of variable and the definition of the $q$-scale function $W^{(q)}(\cdot)$ in Eqn.\ (\ref{W function}), it follows that\begin{equation}
\label{equation J3}
\int_{z=0}^\infty e^{-\Psi(q_{m(j,l)})z}W^{(q_{k})}(z-y){\rm d}z=e^{-\Psi(q_{m(j,l)})y}\frac{1}{q_{m(j,l)}-q_{k}}.
\end{equation}
Substituting Eqn.\ (\ref{equation J3}) into Eqn.\ (\ref{integral J_3}), in combination with the use of Remark \ref{second remark} in the second step, we find
\begin{eqnarray} 
{\mathcal J}_3 &=& -c^{(2^lj-2^{l-1}+1,k)}\Psi(q_{m(j,l)})e^{-\Psi(q_{m(j,l)})y}\prod_{\mathclap{\substack{i=1,\\ i\neq m(j,l)}}}^{k}q_i\prod_{i=l}^{m(j,l)}\frac{1}{q_{i}-q_{i+1}}\prod_{i=m(j,l)+1}^{k-1}\frac{1}{q_{m(j,l)}-q_{i+1}} \nonumber \\
%&=& c^{(2^lj-2^{l-1}+1,k+1)}\Psi(q_{m(j,l)})e^{-\Psi(q_{m(j,l)})y}\prod_{\mathclap{\substack{i=1,\\ i\neq m(j,l)}}}^{k}q_i\prod_{i=l}^{m(j,l)}\frac{1}{q_{i}-q_{i+1}}\prod_{i=m(j,l)+1}^{k-1}\frac{1}{q_{m(j,l)}-q_{i+1}} \nonumber \\
&=&  L_{(2^lj -  2^{l-1}+1,l)}^{(k)}(y), \label{L from J3}
\end{eqnarray}
as desired.

\vb

\noindent
$\circ$ {\it Integral ${\mathcal J}_4$.} \hspace{1mm} 
Before proceeding to the last integral, it is noted that from Fig.~\ref{tree graph coefficients spectrally negative} we observe that when we are at node $L_{(2^lj-2^{l-1}+1,l)}^{(k-1)}$, for some $l=1,\ldots,k-1$, $j=1,\ldots,2^{k-l-1}$ and row $k-1$, and we move down and right, then we obtain the coefficient $L_{(2^lj-2^{l-1}+1+2^{k-1},l)}^{(k)}$. But this is equivalent to $L_{(2^lj-2^{l-1}+1,l)}^{(k)}$ for $j=2^{k-l-1}+1,\ldots,2^{k-l}$. From the definition of the $Z$-scale function in Eqn.\ (\ref{Z scale function}) and interchanging integrals, it follows that 
\[\int_0^\infty e^{-\Psi(q_{m(j,l)})z}Z^{(q_{k})}(z){\rm d}z=\frac{1}{\Psi(q_{m(j,l)})}\frac{q_{m(j,l)}}{q_{m(j,l)}-q_{k}}.\]
Due to the fact that $j=1,\ldots,2^{k-l-1}$ and $l=1,\ldots,k-1$ we have that $m(j,l)\leq k-1$, this leads to
\begin{equation}
\label{equation J4}
{\mathcal J}_4=c^{(2^lj-2^{l-1}+1,k)}\Psi(q_{k})e^{-\Psi(q_{k})y}\prod_{i=1}^{k-1}q_i\prod_{i=l}^{m(j,l)}\frac{1}{q_{i}-q_{i+1}}\prod_{i=m(j,l)+1}^{k-1}\frac{1}{q_{m(j,l)}-q_{i+1}}.
\end{equation}
The last property we should verify is that, for $l=1,\ldots,k-1$ and $j=1,\ldots,2^{k-l-1}$,
it holds that ${\mathcal J}_4=L_{(2^lj-2^{l-1}+1+2^{k-1},l)}^{(k)}$.
But we observe that this is equivalent to showing that, for $l=1,\ldots,k-1$ and $j=2^{k-l-1}+1,\ldots,2^{k-l}$,
\begin{equation}
\label{equation2 J4}
{\mathcal J}_4=L_{(2^lj-2^{l-1}+1,l)}^{(k)}.
\end{equation}
For the values of $l$ and $j$ we consider, we have that $m(j,l)=k$ and thus, by Remark \ref{third remark}, we see that Eqn.~(\ref{equation J4}) gives Eqn.\ (\ref{equation2 J4}).

\bigskip \noindent
Returning to (\ref{huge expression with 4 integrals}), and using the expressions found in (\ref{J1 integral}), (\ref{J2 integral}), (\ref{L from J3}) and (\ref{equation2 J4}), we find 
\begin{eqnarray*}
 & & \hspace{-5em} \Pb_x(\Q_{T_1+...+T_k}\in{\rm d}y) \\
&=& \Bigg[ c^{(1,k)}\prod_{i=1}^{k}q_i \left( W^{(q_{k})}\star\ldots\star W^{(q_1)}\right)(x-y) \\
 & & + c^{(2,k)}\prod_{i=1}^kq_i\Psi(q_{k})e^{-\Psi(q_{k})y}\left(Z^{(q_{k})}\star\ldots\star W^{(q_1)}\right)(x) \\
& & + \sum_{l=1}^{k-1}\sum_{j=1}^{2^{k-l-1}}\left( L_{(2^lj-2^{l-1}+1,l)}^{(k)}(y)+L_{(2^l(j+2^{k-l-1})-2^{l-1}+1,l)}^{(k)}\right)\left(Z^{(q_l)}\star...\star W^{(q_1)}\right)(x)\Bigg]{\rm d}y.
\end{eqnarray*}
The last expression can be written more compactly, yielding
\begin{eqnarray*}
\Pb_x(\Q_{T_1+\ldots+T_k}\in{\rm d}y) &=& \Bigg[ c^{(1,k)}\prod_{i=1}^{k}q_i \left( W^{(q_{k})}\star...\star W^{(q_1)}\right)(x-y)\\
& & +\sum_{l=1}^{k}\sum_{j=1}^{2^{k-l}}  L_{(2^lj-2^{l-1}+1,l)}^{(k)}(y) \left(Z^{(q_l)}\star\ldots\star W^{(q_1)}\right)(x)\Bigg]{\rm d}y.
\end{eqnarray*}
This concludes the proof of Thm.~\ref{spectrally negative general case density}. 
\end{proof}

\begin{proof}[Proof of Corollary \ref{general triple transform}]
The expression in Corollary \ref{general triple transform} is derived as a straightforward application of the result established in Thm.\ \ref{spectrally negative general case density}. For the triple transform we know that 
\[\int_0^\infty e^{-\beta x}\E_x e^{-\alpha \Q_{T_1+\ldots+T_n}}{\rm d}x = \int_{x=0}^\infty e^{-\beta x}\int_{y=0}^\infty e^{-\alpha y}\Pb_x(\Q_{T_1+...+T_n}\in{\rm d}y){\rm d}x.\]
Using Thm.\ \ref{spectrally negative general case density} we see that 
\begin{equation*}
\int_0^\infty e^{-\beta x}\E_xe^{-\alpha\Q_{T_1+\ldots+T_n}}{\rm d}x = c^{(1,n)}\prod_{i=1}^nq_i\int_{x=0}^\infty\int_{y=0}^\infty e^{-\beta x}e^{-\alpha y}\left(W^{(q_n)}\star...\star W^{(q_1)}\right)(x-y){\rm d}y{\rm d}x
\end{equation*}
\begin{equation}
\label{goal}
+\sum_{l=1}^n\sum_{j=1}^{2^{n-l}}\int_{x=0}^\infty e^{-\beta x}\left(Z^{(q_l)}\star W^{(q_{l-1})}\star\ldots\star W^{(q_1)}(x)\right){\rm d}x\cdot\int_{y=0}^\infty e^{-\alpha y}L_{(2^lj-2^{l-1}+1,l)}^{(n)}(y){\rm d}y,
\end{equation}
where the coefficients $L_{(2^lj-2^{l-1}+1,l)}^{(n)}$ are given in Definition \ref{lambdaexpo}.
We see that we have to work with the following three integrals:
\begin{eqnarray*}
%\label{first type of integral}
{\mathcal K}_1 &=& \int_{x=0}^\infty \int_{y=0}^\infty e^{-\beta x}e^{-\alpha y}\left(W^{(q_n)}\star\ldots\star W^{(q_1)}\right)(x-y){\rm d}y{\rm d}x, \\
%\label{second type of integral}
{\mathcal K}_2 &=& \int_{x=0}^\infty e^{-\beta x}\left(Z^{(q_l)}\star W^{(q_{l-1})}\star\ldots\star W^{(q_1)}(x)\right){\rm d}x, \\
%\label{third type of integral}
{\mathcal K}_3 &=& \int_{y=0}^\infty e^{-\alpha y}L_{(2^lj-2^{l-1}+1,l)}^{(n)}(y){\rm d}y.
\end{eqnarray*}
Relying on the properties of the $q$-scale functions and after some straightforward calculus, we find \[
{\mathcal K}_1 = \frac{1}{\alpha+\beta}\prod_{i=1}^n\frac{1}{\Phi(\beta)-q_i},
\:\:\:\:\:\:\:
{\mathcal K}_2 =  \frac{\Phi(\beta)}{\beta}\prod_{i=1}^l\frac{1}{\Phi(\beta)-q_i}\]
and 
\[{\mathcal K}_3 = c^{(2^lj-2^{l-1}+1,n)}\frac{\Psi(q_{m(j,l)})}{\alpha+\Psi(q_{m(j,l)})}\prod_{i=1,i\neq m(j,l)}^n q_i \prod_{i=l}^{m(j,l)} \frac{1}{q_{i}-q_{i+1}}\prod_{i=m(j,l)+1}^{n-1}\frac{1}{q_{m(j,l)}-q_{i+1}}.\]
We conclude that by substitution of ${\mathcal K}_1,{\mathcal K}_2$, and ${\mathcal K}_3$ in (\ref{goal}), we have established Cor.\ \ref{general triple transform}.
\end{proof}

\section{Conclusion and Discussion} \label{sec:concl}

In this paper we have analyzed the transient behavior of spectrally one-sided L\'evy-driven queues. We considered the joint behavior of  $\Q_{T_1},\Q_{T_1+T_2},\ldots,\Q_{T_1+\ldots+T_n}$ where $T_i$ is exponentially distributed with parameter $q_i$, and we specifically focused on $\Q_{T_1+\ldots+T_n}$. From the main results it follows that this transient behavior obeys an elegant and appealing tree structure. Interestingly, some numerical illustrations showed that $\E_x \Q_t$ is first decreasing in $t$ and then converges to the steady-state workload from below in case $x$ is chosen `slightly' above the stationary workload.

We have restricted ourselves to analyzing $\Q_{T}$ with $T$ distributed as the sum of $n$ independent exponential random variables, but our result is readily extended to that of $\Q_T$ with $T$ obeying a {\it Coxian} distribution. This is a particularly useful fact, as any distribution on the positive half line 
can be approximated arbitrarily closely  by a sequence of Coxian distributions, see e.g.\ \cite[Section III.4]{Asm}. In more detail, the analysis looks as follows. Consider the situation that $T$ follows a Coxian distribution with $n$ phases; we let the length of phase $i$ be drawn from an exponential distribution with parameter $q_i$, and we let the probability of moving from phase $i$ to $i+1$  be $p_i$ (with the convention that $p_n = 0$). Then, for the spectrally-positive case, 
\[  \E_x e^{-\alpha \Q_T} = \sum_{k=1}^n  (1-p_k) \prod_{i=1}^{k-1} p_i   \cdot \E_x  e^{-\alpha \Q_{T_1 + \ldots + T_k}}, \] 
where $\E_x  e^{-\alpha \Q_{T_1 + \ldots + T_k}}$ is as obtained in Thm.~\ref{MainTheorem}. The density in the spectrally-negative case follows by a similar argument. 

To conclude, we like to mention some topics that are of interest for future investigation. Although the class of Coxian distributions for the epoch $T$ is sufficiently rich, it might of interest to study the behavior of $\Q_T$ if $T$ has a general phase-type distribution. Specifically, we did not explicitly derive the results in case some parameters $q_i$ are identical. This follows as a direct application of  l'H\^{o}pital's rule, but the expressions tend to become cumbersome. 
%The case of phase-type distribution is also of interest for approximating $\Q_t$ in the spirit of Tijms~\cite{Tijms2007}. 
Another open question concerns the transient behavior for spectrally two-sided L\'evy processes. 
Finally, we expect that the transient analysis presented here may be applicable in inference procedures, to estimate the queue's L\'evy input process from a finite number of successive workload observations. 

\end{document}